\documentclass{amsart}

\usepackage{amsmath}
\usepackage{amsthm}
\usepackage{amsopn}
\usepackage{amssymb}
\usepackage[all]{xy}

\newcommand{\mc}[1]{\mathcal{#1}}
\newcommand{\mb}[1]{\mathbb{#1}}
\newcommand{\mr}[1]{\mathrm{#1}}
\newcommand{\mbf}[1]{\mathbf{#1}}
\newcommand{\mit}[1]{\mathit{#1}}

\newcommand{\bra}[1]{\langle #1 \rangle}
\newcommand{\br}[1]{\overline{#1}}

\newcommand{\td}[1]{\widetilde{#1}}

\newcommand{\ZZ}{\mathbb{Z}}
\newcommand{\RR}{\mathbb{R}}
\newcommand{\CC}{\mathbb{C}}
\newcommand{\QQ}{\mathbb{Q}}

\newcommand{\FF}{\mathbb{F}}
\newcommand{\GG}{\mathbb{G}}
\newcommand{\MS}{\mathbb{S}}
\newcommand{\AF}{\mathbb{A}}

\newcommand{\TAF}{\mathrm{TAF}}
\newcommand{\TMF}{\mathrm{TMF}}
\newcommand{\tmf}{\mathrm{tmf}}
\newcommand{\Sh}{\mathrm{Sh}}
\newcommand{\disc}{\mathrm{disc}}

\theoremstyle{definition}

 \newtheorem{thm}[equation]{Theorem}
 \newtheorem{cor}[equation]{Corollary}
 \newtheorem{lem}[equation]{Lemma}
 \newtheorem{prop}[equation]{Proposition}
 
 \newtheorem{ex}[equation]{Example}
 
 \newtheorem{rmk}[equation]{Remark}
\newtheorem*{thm*}{Theorem}
\newtheorem*{cor*}{Corollary}
\newtheorem*{lem*}{Lemma}
\newtheorem*{prop*}{Proposition}
\newtheorem*{defn*}{Definition}
\newtheorem*{ex*}{Example}
\newtheorem*{exs*}{Examples}
\newtheorem*{rmk*}{Remark}
\newtheorem*{claim*}{Claim}

\numberwithin{equation}{section}
\numberwithin{figure}{section}

\DeclareMathOperator{\Hom}{Hom}

\DeclareMathOperator{\Br}{Br}
\DeclareMathOperator{\Gal}{Gal}
\DeclareMathOperator{\Art}{Art}
\DeclareMathOperator{\Inv}{Inv}
\DeclareMathOperator{\Res}{Res}
\DeclareMathOperator{\End}{End}
\DeclareMathOperator{\Aut}{Aut}
\DeclareMathOperator{\Spec}{Spec}
\DeclareMathOperator{\Spf}{Spf}
\DeclareMathOperator{\Lie}{Lie}
\DeclareMathOperator{\Ind}{Ind}
\DeclareMathOperator{\Tr}{Tr}

\DeclareMathOperator*{\colim}{colim}

\title{Higher real $K$-theories and topological automorphic
forms}

\author[M.~Behrens]{M.~Behrens$\sp 1$}
\address{
Dept. of Mathematics \\
M.I.T. \\
Cambridge, MA  02139, U.S.A.
}
\author[M.J.~Hopkins]{M.J.~Hopkins$\sp 2$}
\address{
Dept. of Mathematics \\
Harvard University \\
Cambridge, MA  02138, U.S.A.
}

\date{\today}

\begin{document}

\begin{abstract}
Given a maximal finite subgroup $G$ of the $n$th 
Morava stabilizer group at a prime $p$, 
we address the question: is the associated higher real
$K$-theory $EO_n$ a summand of the $K(n)$-localization of a
$\TAF$-spectrum associated to a unitary similitude group of type $U(1,n-1)$?  
We answer this question in the affirmative for $p \in \{2, 3, 5, 7\}$ and
$n = (p-1)p^{r-1}$ for a maximal finite subgroup containing an element of
order $p^r$.  We answer the question in the negative for all other odd
primary cases.  In all odd primary cases, we to give an explicit
presentation of a global division algebra with involution in which the
group $G$ embeds unitarily.
\end{abstract}

\maketitle

\footnotetext[1]{The first author was supported by the NSF, the Sloan
Foundation, and DARPA.}
\footnotetext[2]{The second author was supported by the NSF and DARPA.}

\tableofcontents

\section{Introduction}

For a prime $p$,
let $E_n$ denote the Morava $E$-theory spectrum 
associated with the Honda formal
group $H_n$ of height $n$ over $\bar{\FF}_p$, so that
$$ \pi_* E_n = W(\bar{\FF}_p)[[u_1, \ldots, u_{n-1} ]]. $$
The second author, with P.~Goerss and H.~Miller,
showed that $E_n$ is an $E_\infty$ ring spectrum, and that 
the $n$th Morava stabilizer group
$$ \MS_n = \Aut(H_n) $$
acts on $E_n$ by $E_\infty$ ring maps \cite{Rezk},
\cite{GoerssHopkins}.

One of the original motivations for producing this action was that, given a
maximal finite subgroup $G$ of $\MS_n$, the spectrum
$$ EO_n := E_n^{hG} $$
would more effectively detect $v_n$-periodic homotopy groups of
spheres, generalizing the phenomenon that 
the spectrum $KO$ detects the $2$-primary image of $J$ more effectively
than the spectrum $KU$.
In fact, for $p = 2$, there is an equivalence 
\begin{equation}\label{eq:EO1}
KO_{K(1)} \simeq EO_1^{h\Gal(\bar{\FF}_p/\FF_p)}. 
\end{equation}
It is for this reason that the spectra $EO_n$ are regarded as ``higher real
$K$-theories.''  

For the primes $2$ and $3$, there are equivalences
\begin{equation}\label{eq:EO2}
\TMF_{K(2)} \simeq EO_2^{h\Gal(\bar{\FF}_p/\FF_p)} 
\end{equation}
where $\TMF$ denotes the spectrum of topological modular forms.
Using a theorem of J.~Lurie, the first author and T.~Lawson constructed
$p$-complete spectra $\TAF_{GU}(K)$ 
of topological automorphic forms associated to unitary similitude
groups over $\QQ$ of signature $(1,n-1)$ and compact open subgroups $K
\subset GU(\AF^{p,\infty})$ \cite{taf}.  
The spectrum $\TAF_{GU}(K)$ arises from a Shimura stack $\Sh(K)$, in the
same manner that the spectrum $\TMF$ arises from the moduli stack of
elliptic curves.
The spectra $\TAF_{GU}(K)$ detect
$E(n)$-local phenomena in the same sense that the spectrum $\TMF^\wedge_p$
detects $E(2)$-local phenomena.  This is made precise as follows: there is
an equivalence
\begin{equation*}
\TAF_{GU}(K)_{K(n)} \simeq 
\left(
\prod_{x \in \Sh^{[n]}(K)(\bar{\FF}_p)}
E_n^{h\Aut(x)}
\right)^{h\Gal(\bar{\FF}_p/\FF_p)}
\end{equation*}
where $\Sh^{[n]}(K)$ is the (non-empty) 
finite $0$-dimensional substack of $\Sh(K)$
where the associated formal group has height $n$, and the automorphism
groups $\Aut(x)$ are finite subgroups of $\MS_n$ \cite[Cor.~14.5.6]{taf}.  
Given equivalences (\ref{eq:EO1}) and
(\ref{eq:EO2}) it is natural to ask:
\begin{quote}
For a given prime $p$ and chromatic level $n$, does there exist a pair
$(GU, K)$ so that there exists an $x \in \Sh^{[n]}(K)$ for which $\Aut(x)$
is a maximal finite subgroup of $\MS_n$?
\end{quote}
The purpose of this paper is to provide answers to this question.

For large $n$, the notation $EO_n$ is ambiguous,
because there exist multiple non-isomorphic choices of maximal finite
subgroups $G$.  Hewett \cite{Hewett} determined precisely the collection of
maximal finite subgroups of $\MS_n$.  In particular, he showed that if $r$
is the largest exponent so that $G$ contains an element of order $p^r$, then
the isomorphism class of $G$ is determined by the pair $(n,r)$.  If $p$ is odd and $G$ is non-abelian, then as a subgroup of the group of units of the associated division algebra, it is unique up to conjugation \cite[Prop.~6.11]{Hewett}.  In this
paper, we will denote such a subgroup $G_r$.  We summarize our results in
the following theorem.

\begin{thm}\label{thm:mainthm}
Suppose that $p$ is odd and $n = (p-1)p^{\alpha-1}m$ for a positive integer $m$.  Then there exists a
pair $(GU,K)$ whose associated Shimura stack has a height $n$ mod $p$
point with automorphism group $G_\alpha$ if and only if $p \in \{ 3, 5,
7\}$ and $m = 1$.  If $p = 2$, and $n = 2^{r-1}$ for $r >2$, we can also
realize $G_{r-1}$ as an automorphism group.
\end{thm}

\begin{rmk}
In the above theorem, 
the case of $p = 2$ and $n = 1$ is handled by $KO$, and 
the case of $p = 2$ and $n = 2$ is the
$\TMF$ case.  However, both of these cases may also be viewed as instances
of $\TAF$: see \cite[Ch.~15]{taf} for the $n=1$ case, and
\cite{BehrensLawson} for the $n = 2$ case.
\end{rmk}

We also prove the following algebraic theorem, which could be regarded as a
kind of global analog of Hewett's results.

\begin{thm}\label{thm:mainthm2}
Suppose that $p$ is odd, and that $n = (p-1)p^{\alpha-1}m$.  Then there
exists a global division algebra $D$ with positive 
involution $\dag$ of the second
kind, so that:
\begin{enumerate}
\item we have $[D:F] = n^2$, where $F$ is the center of $D$, 
\item there is a prime $x$ of the fixed field $F^{\dag=1}$ which splits as
$yy^\dag$ in
$F$, so that $\Inv_{y} D = 1/n$,
\item there is a $\dag$-unitary 
embedding of the maximal finite subgroup $G_\alpha
\subset \MS_n$ in $D^\times$. 
\end{enumerate}
Since $G_\alpha$ is a maximal finite subgroup in the completion $D_y^\times$, 
it is
necessarily a maximal finite subgroup in $D^\times$.
\end{thm}

As we point out in Section~\ref{sec:moreShimura}, using the theorem above, one can adapt our methods to argue that for \emph{every} odd prime $p$ and chromatic height $n = (p-1)p^{\alpha-1}m$, there exists a unitary Shimura stack with a mod $p$ point with $G_\alpha$ as its automorphism group.  The associated unitary group is defined over a totally real field $F^+$.  However, unless $p \le 7$, this totally real field is \emph{not} $\QQ$.  As \cite{taf} only associates $\TAF$-spectra to Shimura stacks of certain unitary groups over $\QQ$, such global manifestations of Hewett subgroups have no  obvious topological realization.

\subsection*{Organization of the paper}$\quad$

In Section~\ref{sec:taf}, we give an overview of unitary Shimura varieties
of type $(1,n-1)$, and the associated spectra of topological automorphic
forms.  We explain how the automorphism groups of height $n$ mod $p$ points
can be computed using division algebras with involution.

In Section~\ref{sec:explicit}, we give an overview of  
presentations of division
algebras in terms of cohomological data.  In the case of a division algebra
over a local number field, we explain how to make these presentations
explicit using local class field theory.  The material in this section is
essentially well known.

In Section~\ref{sec:Hewett}, we use the explicit presentations of
Section~\ref{sec:explicit} to simultaneously review and 
give a new perspective on Hewett's
maximal finite subgroups $G_\alpha$ in $\MS_n$.

In Section~\ref{sec:negative}, we give one direction of
Theorem~\ref{thm:mainthm}, by showing that $G_\alpha$ cannot be an
automorphism group of a height $n$ mod $p$ point for $p$ odd and $n =
(p-1)p^{\alpha-1}m$ unless $p \le 7$ and $m = 1$.

In Section~\ref{sec:global}, we use global class field theory to give
global analogs of the explicit presentations of division algebras given in
Section~\ref{sec:explicit}.  We use these explicit presentations to embed
the Hewett subgroups $G_\alpha$ into global division algebras.

In Section~\ref{sec:involutions}, we produce involutions on the division
algebras of Section~\ref{sec:global} which act on the finite subgroup
$G_\alpha$ by inversion.  The existence of these involutions completes the
proof of Theorem~\ref{thm:mainthm2}.

In Section~\ref{sec:shaut}, we assume that $p \in \{3, 5, 7\}$ and $n =
(p-1)p^{r-1}$, and show that there is a simple choice of hermitian form of
signature $(1,n-1)$ whose associated Shimura stack has a height $n$ mod $p$
point with automorphism group $G_r$.  This completes the odd primary
cases of Theorem~\ref{thm:mainthm}.

In Section~\ref{sec:rmks}, we give some concluding remarks.  
\begin{itemize}
\item We explain how
the results of this paper relate the Hopkins-Gorbounov-Mahowald approach to
$EO_{p-1}$ to the theory of topological automorphic forms.  
\item 
We explain that for odd primes, 
our results also allow one to realize all of the groups $G_\alpha$ as 
automorphism
groups of some unitary Shimura stack --- but these Shimura stacks do not
have known topological applications (i.e., they are not of type $(1,n-1)$).
\item
We suggest that our explicit global division algebras could shed light on
$EO_{n}$-resolutions of the $K(n)$-local sphere.
\item We explain the applicability of our results to the problem of
producing connective analogs of $EO_n$.
\item We explain that our results combine with a theorem of Mark Hovey to
prove that, at least in some cases, the $\TAF$-spectra do not admit an
orientation for \emph{any} connective cover of $O$.
\item We show that for any prime $p$, and $n = (p-1)p^{r-1}$, there exists
a unitary Shimura stack of type $(1,n-1)$ with a height $n$ mod $p$ point
whose automorphism group contains an element of order $p^r$.  
\item Specializing the previous observation to the prime $2$, we explain
how this gives an automorphism group isomorphic to $G_r$, 
in the case where $p = 2$ and $n = 2^{r-1}$.  This completes the proof of
Theorem~\ref{thm:mainthm}.
\end{itemize}

\subsection*{Conventions}$\quad$

In this paper, we shall use the following notation.
\begin{itemize}
\item $\zeta_n = $ primitive $n$th root of unity.
\item $K[p^n] = $ cyclotomic $\ZZ/p^n$-extension of $K$ (i.e. the fixed
field of $K(\zeta_{p^{n+1}})$ by the subgroup of $(\ZZ/p^{n+1})^\times =
\Gal(K(\zeta_{p^{n+1}})/K)$ of order $p-1$).  
Here, we are assuming $K$ satisfies
$[K(\zeta_{p^{n+1}}):K] = (p-1)p^{n}$. 
\item $\Art$ denotes the local/global Artin map.
\item $H^*(L/K) = $ the Galois cohomology group $H^*(\Gal(L/K), L^\times)$.
\item $\AF = $ the rational adeles.  If $S$ is a set of places, $\AF^S$
denotes the adeles away from $S$, whereas $\AF_S$ denotes the adeles at
$S$. For a global number field $K$, $\AF_K$ denotes the $K$-adeles.
\item $\mb{I_K} = $ the $K$-ideles $\AF_K^\times$.
\end{itemize}

\subsection*{Acknowledgments}$\quad$

The first author completed a portion of this work while visiting Harvard
university, and is thankful for their hospitality.  The authors would also
like to thank Tyler Lawson and Niko Naumann, for their constructive 
comments on an
earlier version of this paper.  The authors also thank the referee for suggesting valuable expository improvements.

\section{Overview of topological automorphic forms}\label{sec:taf}

We review the theory of topological automorphic forms presented in
\cite{taf}.  (The $p$-integral models of the Shimura varieties discussed here are special cases of those constructed and studied in \cite{Kottwitz}.)

Fix a prime $p$ and consider the following initial data: 
\begin{align*}
F  = & \: \text{quadratic imaginary extension of $\QQ$ in which $p$ splits as
$u\bar{u}$,} \\
B = & \: \text{central simple $F$-algebra of dimension $n^2$ which splits at $u$ and
$\bar{u}$,} \\
* = & \: \text{positive involution of the second kind on $B$,} \\
\mc{O}_{B,(p)} = & \: \text{$*$-invariant maximal $\mc{O}_{F,(p)}$-order of
$B$,}
\\
V = \: & \text{$B$ module of rank $1$,} \\
\bra{-,-} = \: & \text{$\QQ$-valued $*$-hermitian alternating
form of signature $(1, n-1)$,}
\\
\epsilon = & \text{rank $1$ idempotent of $\mc{O}_{B,u} \cong
M_n(\ZZ_p)$.}
\end{align*}
Let $\iota$ denote the involution on
$\End_B(V)$, defined by 
$$ \bra{\alpha v, w} = \bra{v, \alpha^\iota w}. $$  
Let $GU = GU_V$ be the associated unitary similitude group over $\QQ$, 
with $R$-points
\begin{align*}
GU(R) = & \{ g \in \End_B(V) \otimes_\QQ R \: : \: \bra{gv, gw} = \nu(g)
\bra{v,w}, \: \nu(g) \in R^\times \} \\
= & \{ g \in \End_B(V) \otimes_\QQ R \: : \: g^\iota g \in R^\times \}
\end{align*}
(so that $\nu(g) = g^\iota g \in R^\times$).

We let $V^{p,\infty}$ denote $V \otimes \AF^{p,\infty}$.
For every compact open subgroup 
$$ K \subset GU(\AF^{p,\infty}) $$
there is a Deligne-Mumford stack $\Sh(K)/ \Spec(\ZZ_p)$.  For a
locally noetherian connected $\ZZ_p$-scheme $S$,
and a geometric point $s$ of $S$,
the $S$-points of $\Sh(K)$ are the groupoid whose objects are tuples
$(A,i,\lambda,[\eta]_{K})$, with:
\begin{center}
\begin{tabular}{lp{19pc}}
$A$, & an abelian scheme over $S$ of dimension $n^2$,\\
$\lambda: A \rightarrow A^\vee$, & a $\ZZ_{(p)}$-polarization, \\
$i: \mc{O}_{B,(p)} \hookrightarrow \End(A)_{(p)}$, & an inclusion of
rings, such that the $\lambda$-Rosati involution is compatible with
conjugation, \\
$[\eta]_K$, & a $\pi_1(S,s)$-invariant $K$-orbit
of $B$-linear similitudes: \\
& \qquad $ \eta: (V^{p,\infty}, \bra{-,-}) \xrightarrow{\cong}
(V^p(A_s), \bra{-,-}_\lambda), $
\end{tabular}
\end{center}
subject to the following condition:
\begin{equation}\label{eq:condition}
\text{the coherent sheaf $\Lie A \otimes_{\mc{O}_{F,p}} \mc{O}_{F,u}$ is
locally free of rank $n$.}
\end{equation}
Here, since $S$ is a $\ZZ_p$-scheme, the action of $\mc{O}_{F,(p)}$ on
$\Lie A$
factors through the $p$-completion $\mc{O}_{F,p}$.

The morphisms 
$$ (A, i, \lambda, \eta) \rightarrow (A', i', \lambda',\eta') $$
of the groupoid of $S$-points of $\Sh(K)$ are the
prime-to-$p$ quasi-isogenies of abelian schemes 
$$ \alpha : A \xrightarrow{\simeq} A' $$
such that
\begin{alignat*}{2}
\lambda & = r\alpha^\vee \lambda' \alpha, 
\quad && 
r \in \ZZ_{(p)}^\times, \\
i'(z)\alpha & = \alpha i(z), 
\quad &&
z \in \mc{O}_{F,(p)}, \\
[\eta']_K & = [\eta \circ \alpha_*]_K.
\quad &&
\end{alignat*}

\begin{rmk}
If the algebra $B$ is split, then the moduli interpretation of the points
of $\Sh(K)$ may be simplified.  Namely, the idempotent $\epsilon$ may be
extended to a rank $1$ idempotent on $B$, and the moduli of $B$-linear
abelian schemes of dimension $n^2$ becomes Morita equivalent to the moduli of
$F$-linear abelian schemes of dimension $n$.
Thus, in this case, the $S$-points of $\Sh(K)$ could instead be taken to be 
a groupoid of tuples
$(A,i,\lambda,[\eta]_{K})$, with $(A,i)$ an abelian scheme of dimension $n$
with complex multiplication by $F$.
\end{rmk}

The $p$-completion $\Sh(K)^\wedge_p/\Spf(\ZZ_p)$ is determined by the
$S$-points of $\Sh(K)$ on which $p$ is locally nilpotent.  On such schemes,
the abelian scheme $A$ 
has an $n^2$-dimensional
$p$-divisible group $A(p)$ of height $2n^2$. The composite 
$$ \mc{ O}_{B,(p)} \xrightarrow{i} \End(A)_{(p)}
\to \End(A(p)) $$ 
factors through the $p$-completion
\[
\mc{O}_{B,p} \cong \mc{O}_{B,u} \times \mc{O}_{B,\bar u} \cong 
M_n(\mb Z_p) \times M_n(\mb Z_p).
\]
Therefore, the action of $\mc{ O}_B$ naturally splits $A(p)$ into two
summands, $A(u)$ and $A(\bar u)$, both of height $n^2$.  For such schemes
$S$, Condition~(\ref{eq:condition}) is equivalent to the condition that
$A(u)$ is $n$-dimensional.  This forces the formal group of $A$ to split
into $n$-dimensional and $n(n-1)$-dimensional formal summands.  The action of
$\mc{O}_{B,u}$ on $A(u)$ provides a splitting
$$ A(u) = \epsilon A(u)^n $$
where $\epsilon A(u)$ is a $1$-dimensional $p$-divisible group 
of height $n$.

A theorem of Jacob Lurie \cite[Thm.~8.1.4]{taf} associates to a
$1$-dimensional $p$-divisible group $\GG$ over a 
locally noetherian separated Deligne-Mumford stack
$X/\mathrm{Spec}(\ZZ_p)$ which is locally a universal deformation of all of its
mod $p$ points, a (Jardine fibrant) presheaf
of $E_\infty$-ring spectra $\mc{E}_\GG$ on the site
$(X^\wedge_p)_{et}$.  The presheaf $\mc{E}_{\GG}$ is functorial in $(X,
\GG)$ (the precise statement of this functoriality is given in
\cite[Thm.~8.1.4]{taf}).  

If $(\mbf{A}, \mbf{i}, \pmb{\lambda}, [\pmb{\eta}])$
is the universal tuple over $\Sh(K)$, then the $p$-divisible group
$\epsilon \mbf{A}(u)$ satisfies the hypotheses of Lurie's theorem 
\cite[Sec.~8.3]{taf}.  The 
associated sheaf will be
denoted
$$ \mc{E}_{GU} := \mc{E}_{\epsilon\mbf{A}(u)}. $$
The $E_\infty$-ring spectrum of topological automorphic forms is obtained
by taking the homotopy global sections:
$$ \TAF_{GU}(K) := \mc{E}_{GU}(\Sh(K)^\wedge_p). $$

Let $\epsilon A(u)^0$ denote the formal subgroup of the $p$-divisible
group $\epsilon A(u)$.
Let $\Sh(K)_{\FF_p}$ denote the reduction mod $p$ of $\Sh(K)$, and let
$$ \Sh(K)^{[n]} \subseteq \Sh(K)_{\FF_p} $$
denote the $0$-dimensional substack where the height of the
formal group $\epsilon A(u)^0$ is equal to $n$.

Picking a point $\mbf{A}_0 = (A_0, i_0, \lambda_0, [\eta_0]_K) \in
\Sh(K)^{[n]}(\bar{\FF}_p)$, the
ring of prime-to-$p$ $\mc{O}_B$-linear quasi-endomorphisms of 
$A_0$ is the unique maximal
$\mc{O}_{F,(p)}$-order in a division algebra $D$ with center $F$:
$$ \End_{\mc{O}_{B,(p)}}(A_0)_{(p)} = \mc{O}_{D,(p)} \subset D. $$
We have, for $x$ a finite place of $F$,
$$
\Inv_x D = 
\begin{cases}
\Inv_x B, & x \not\vert p, \\
\frac{1}{n}, & x = u, \\
\frac{n-1}{n}, & x = \bar{u}.
\end{cases}
$$
The prime-to-$p$ polarization $\lambda_0$ of $A_0$ gives rise to an
associated Rosati involution $\dag$ on $D$.  The Rosati involution is
a positive involution of the second kind.
Define the associated unitary similitude group $GU_{\mbf{A}_0}/\mc{\ZZ}_{(p)}$
by
$$ GU_{\mbf{A}_0}(R) = 
\{ g \in \mc{O}_{D,(p)} \otimes_{\ZZ_{(p)}} R \: : \: g^\dag g \in R^\times
\}. $$
Fixing a representative $\eta_0$ of $[\eta_0]$ gives a $B$-linear isomorphism
$$ \eta_0: V^{p,\infty} \rightarrow V^{p}(A_0). $$
The induced $B$-linear action of $\mc{O}_{D,(p)}$ on $V^{p,\infty}$ is an
action by similitudes, and induces an
isomorphism
$$ \xi_{\eta_0} : GU_{\mbf{A}_0}(\AF^{p,\infty}) \xrightarrow{\cong}
GU(\AF^{p,\infty}). $$
Under this isomorphism, the subgroup $K \subset GU(\AF^{p,\infty})$ may be
regarded as a subgroup of $GU_{\mbf{A}_0}(\AF^{p,\infty})$.  We define
$$ \Gamma(K) = GU_{\mbf{A}_0}(\ZZ_{(p)}) \cap K \subset D^\times. $$
Since $K$ is open, the group $\Gamma(K)$ is finite.
Since $GU_{\mbf{A}_0}(\ZZ_p) = \mc{O}^\times_{D,u} = \MS_n$, each of
the finite groups $\Gamma(K)$ are finite subgroups of the Morava
stabilizer group.

\begin{lem}[{\cite[Prop.~14.1.2]{taf}}]
The automorphism group of the point $\mbf{A}_0$ is given by
$$ \Aut(\mbf{A}_0) = \Gamma(K). $$
\end{lem}

Let $E_n$ be the Morava $E$-theory associated to a height $n$ formal
group over $\bar{\FF}_p$. 

\begin{thm}[{\cite[Cor.~14.5.6]{taf}}]
There is an equivalence
$$ \TAF_{GU}(K)_{K(n)} \simeq 
\left(
\prod_{[g] \in GU_{\mbf{A}_0}(\ZZ_{(p)})
\backslash GU(\AF^{p,\infty})/K}
E_n^{h\Gamma(gKg^{-1})}
\right)^{h\Gal(\bar{\FF}_p/\FF_p)}.
$$
\end{thm}

Therefore, the problem of realizing $EO_n$'s in the
$K(n)$-localization of a $\TAF$-spectrum amounts to determining which
maximal finite subgroups of $\MS_n$ arise as a $\Gamma(K)$ for some
choice of $F$, form of $GU$, and subgroup $K$.

\section{Explicit division algebras}\label{sec:explicit}

Let $K$ be a local or global number field.
Following Serre \cite{Serre}, we make explicit the isomorphism
$$ \Br(K) \cong H^2(\Gal(\bar{K}/K); \bar{K}^\times). $$
We will regard elements of $H^2(G,A)$ as corresponding to
extensions
$$ 1 \rightarrow A \rightarrow E \rightarrow G \rightarrow 1.
$$
Since we have
$$ H^2(\Gal(\bar{K}/K); \bar{K}^\times) = 
\colim_{\substack{M/K \\ [M:K] < \infty}}
H^2(\Gal(M/K);
M^\times),
$$
every element of $H^2(\Gal(\bar{K}/K); \bar{K}^\times)$
originates in $H^2(\Gal(M/K), M^\times)$ for some finite
extension $M$ of $K$.

Given a central simple algebra $B/K$ with $[B:K] = n^2$, we
choose a maximal subfield $M \subseteq B$ containing $K$,
so that $[M:K] = n$.  Define a group $E$ by
$$ E = \{ x \in B^\times \: : \: xMx^{-1} = M \}. $$
Then the short exact sequence
$$ 1 \rightarrow M^\times \rightarrow E \rightarrow \Gal(M/K)
\rightarrow 1 $$
gives the desired class in $H^2(\Gal(M/K); M^\times)$.

Conversely, suppose that we are given an element of
$H^2(\Gal(M/K), M^\times)$ corresponding to an extension
$$ 1 \rightarrow M^\times \rightarrow E \rightarrow \Gal(M/K)
\rightarrow 1, $$
we may express the corresponding central simple algebra as
$$
B = \ZZ[E] \otimes_{\ZZ[M^\times]} M. 
$$

We explain how to make this construction explicit in the case where $K$ is
a finite extension of $\QQ_p$ and $M$ is a cyclic extension of $K$.
Let
$$ \Art_{M/K} : K^\times/N(M^\times) \xrightarrow{\cong}
\Gal(M/K) $$
be the local Artin map.  

\begin{prop}\label{prop:Ba}
Let $K$ be a finite extension of $\QQ_p$, and suppose that 
$M$ is a degree $n$ cyclic extension of $K$.
Fix an injection
$$ \chi: \Gal(M/K) \hookrightarrow \QQ/\ZZ, $$
and let $\sigma \in \Gal(M/K)$ be the unique element satisfying
$$ \chi(\sigma) = 1/n. $$
Then for each element $a \in K^\times$ there is an extension
$$ 1 \rightarrow M^\times \xrightarrow{i} E_{a} \xrightarrow{j} \Gal(M/K)
\rightarrow 1
$$
such that 
\begin{enumerate}
\item there is a lift of $\sigma$ to $\td{\sigma} \in E_a$ such that
$\td{\sigma}^n = i(a)$, and
\item the corresponding central simple algebra $B_{[a]}$ has invariant
$$ \Inv B_{[a]} = \chi(\Art_{M/K}(a)). $$
\end{enumerate}
\end{prop}

\begin{proof}
Let $G = \Gal(M/K)$.
Consider the cup product pairing
$$ H^0(G, M^\times) \times H^2(G, \ZZ) \xrightarrow{\cup}
H^2(G, M^\times) \xrightarrow[\cong]{\Inv} \QQ/\ZZ, $$
and let $\delta$ denote the connecting homomorphism in the long exact
sequence 
$$ \Hom(G, \QQ/\ZZ) = H^1(G, \QQ/\ZZ) \xrightarrow{\delta} H^2(G, \ZZ). $$
Then we have the following formula \cite[Prop.~4.1]{Milne}:
\begin{equation}\label{eq:Serre}
\chi(\Art(a)) = \Inv(a \cup \delta(\chi)). 
\end{equation}
We may explicitly compute a cochain representative for $\delta(\chi)$.  Let 
$$ \td{\chi}: G \rightarrow \QQ $$
be the unique lift of $\chi$ such that the values of $\td{\chi}$ lie in the
interval $[0,1)$.  Then the cohomology element $\delta(\chi)$ is
represented by the normalized $\ZZ$-valued $2$-cochain
$$ \phi(g_1, g_2) = \td{\chi}(g_2) - \td{\chi}(g_1g_2) + \td{\chi}(g_1). $$
In particular, we have
\begin{align*}
\phi(\sigma^i, \sigma^j) & = 0 \quad \text{if $i + j < n$}, \\
\phi(\sigma^{n-1}, \sigma) & = 1. 
\end{align*}
The cup product $a \cup \delta(\chi)$ is represented by the normalized
$M^\times$-valued $2$-cochain
$$ \phi_a(g_1, g_2) = a^{\phi(g_1, g_2)}. $$
Associated to the $2$-cochain $\phi_a$ is the extension
$$ E_a = M^\times \times G $$
with multiplication
$$ (x_1, g_1)\cdot (x_2, g_2) = (x_1 \cdot {}^{g_1}x_2 a^{\phi{g_1,g_2}},
g_1g_2) $$
We define 
$$ \td{\sigma} = (1,\sigma) \in E_a $$
and compute:
\begin{align*}
\td{\sigma}^n & = (1,\sigma)^{n-1}(1,\sigma) \\
& = (1, \sigma^{n-1})(1,\sigma) \\
& = (a^{\phi(\sigma^{n-1},\sigma)}, \sigma^n) \\
& = (a, 1).
\end{align*}
By (\ref{eq:Serre}), the associated central simple algebra has invariant
$$ \Inv(B_{[a]}) = \chi(\Art(a)). $$
\end{proof}

\begin{rmk}
The algebra $B_{[a]}$ admits the presentation (as a non-commutative
$K$-algebra)
$$ B_{[a]} = M\bra{S}/(S^n = a, Sx = x^\sigma S, x \in M). $$
\end{rmk}

\begin{rmk}\label{rmk:Ba}
The analog of Proposition~\ref{prop:Ba} holds in the
archimedean case as well.  The only non-trivial case to discuss is the case
where $K = \RR$,  $M = \CC$, and $n =
2$.  The character $\chi$ and the element $\sigma$
are uniquely determined.  The algebra $B_{[a]}$ is either $\mb{H}$ or $M_2(\CC)$, 
depending on whether $a$ is negative or positive, that is to say, depending
on the image of $a$ under the Artin map
$$ \Art_{\CC/\RR}: \RR^\times/N(\CC^\times) \xrightarrow{\cong}
\Gal(\CC/\RR) \cong C_2. $$
\end{rmk}

\begin{ex}\label{ex:D1n}
Let $K = \QQ_p$
and let $M = \QQ_{p^n}$ denote the unique
unramified extension of $\QQ_p$ of degree $n$.  
We have.
$$ \Gal(\QQ_{p^n}/\QQ_p) = C_n = \bra{\sigma} $$
where we fix a generator $\sigma = \Art(p)$.  The Artin map is normalized
so that $\sigma$ is a lift of the Frobenius on $\FF_{p^n}$.
Then the central simple algebra $B_{[p^i]}$ has invariant $i/n$, and
presentation:
$$ B_{[p^i]} = \QQ_{p^n}\bra{S}/(S^n = p^i, Sx = x^\sigma S, x \in \QQ_{p^n}). $$
This may be compared to \cite[A2.2.16]{Ravenel}.
\end{ex}

We end this section by giving a canonical presentation of $M_k(B_{[a]})$
given our presentation of $B_{[a]}$.  Let
$$ (M^\times)^k \cong \Ind_{C_{n}}^{C_{nk}} M^\times $$
be the induced $C_{nk}$-module.  The isomorphism above is given so that for
a generator $\sigma'$ of $C_{nk}$, and $\sigma = (\sigma')^k$ a corresponding
generator of the subgroup $C_{k} \le C_{nk}$, we have, for $(m_i) \in
(M^\times)^k$,
$$ \sigma' \cdot (m_1, \ldots, m_k ) = (m_2, \ldots, m_{k}, 
m_1^\sigma). $$
Then the algebra $M_k(B_{[a]})$ admits a presentation
$$ M_k(B_{[a]}) \cong 
M\bra{S}/(S^{nk} = (a,\ldots, a), Sx = x^{\sigma'}S, x \in M^k). $$
The subring $M^k$ is a maximal commutative subalgebra in $B_{[a]}$.
Letting $E'$ denote the normalizer of $(M^\times)^k \subseteq
M_k(B_{[a]})^\times$, we have a short exact sequence
$$ 1 \rightarrow (M^\times)^k \rightarrow E' \rightarrow C_{nk} \rightarrow
1. $$
The extension $E'$ is classified by the image of the cohomology class
$[E_a]$ under the Shapiro isomorphism
$$ H^2(C_n, M^\times) \cong H^2(C_{nk}, \Ind_{C_{n}}^{C_{nk}} M^\times )
\cong H^2(C_{nk}, (M^\times)^k). $$

\section{Elementary presentation of Hewett
subgroups}\label{sec:Hewett}

In this section we explain how the theory of the previous
section may be used to understand the finite subgroups of local
division algebras studied by Hewett \cite{Hewett}.

Let $B$ be a central division algebra over $K$.
Throughout this paper, our technique for constructing finite 
subgroups of $B^\times$ will be construct extensions $G$ of finite subgroups $N$ of 
the group of units of a maximal
subfield $M$ containing $K$ of the following form.
$$
\xymatrix{
1 \ar[r] 
& N \ar@{^{(}->}[d] \ar[r]
& G \ar@{^{(}->}[d] \ar[r] 
& G/N \ar@{^{(}->}[d] \ar[r] 
& 1
\\
1 \ar[r]
& M^\times \ar[r]
& E \ar[r] 
& \Gal(M/K) \ar[r]
& 1
}
$$
Here $E \le B^\times$ is the normalizer of $M^\times$, as in
Section~\ref{sec:explicit}.

Let $p$ be an odd prime, fix $r > 0$, and let 
$$ n = (p-1)p^{r-1}k $$
with $k$ coprime to $p$.
Let $D$ be the central division algebra over $\QQ_p$ of
invariant $1/n$.  
Let $\mc{O}_D$ be the unique maximal order
of $D$, and let
$$ \MS_n = \mc{O}_D^{\times} $$
denote the $n$th Morava stabilizer group.

Hewett \cite{Hewett} showed that the isomorphism
classes of finite subgroups of $\MS_n$ are given by
$$ \{ G_\alpha \: : \: \text{ $0 \le \alpha \le r$} \} $$
such that $G_\alpha$ contains an element of maximal $p$-power order
$p^\alpha$.  We give an elementary construction of these subgroups.

For $\alpha = 0$, the group $G_0$ is cyclic of order $p^n-1$.  Explicitly
it can be taken to be the embedding
$$ G_0 \cong \FF_{p^n}^\times \hookrightarrow \ZZ^\times_{p^n} \hookrightarrow
\mc{O}_D^\times. $$
Here $\ZZ_{p^n} = W(\FF_{p^n})$ is the ring of integers of $\QQ_{p^n}$.

For $0 < \alpha \le r$, the group $G_\alpha$ is metacyclic, with
presentation
$$ G_\alpha = \bra{a,b \: : \: a^{p^\alpha(p^{m}-1)} = 1,
bab^{-1} = a^t, b^{p-1} = a^{p^\alpha}}, $$
where 
\begin{enumerate}
\item $t$ is an integer whose image in $(\ZZ/p^\alpha)^\times$ 
has order $p-1$, and
\item $m = kp^{r-\alpha}$.
\end{enumerate}
The group $G_\alpha$ fits into a short exact
sequence
$$ 1 \rightarrow C_{p^\alpha(p^{m}-1)} \rightarrow G_\alpha
\rightarrow C_{p-1} \rightarrow 1 $$
where the group $C_{p^\alpha(p^m-1)}$ is generated by $a$, and the group
$C_{p-1}$ is generated by the image of $b$.

We now give an explicit, elementary embedding of $G_\alpha$ into
$\mc{O}_D^\times$.  Consider the following tower of abelian Galois extensions.
$$
\xymatrix{
M = \QQ_{p^m}(\zeta_{p^{\alpha}}) \phantom{= M} \ar@{-}[d]^{C_{p-1}} \\
L = \QQ_{p^m}[p^{\alpha-1}] \phantom{= L} \ar@{-}[d]^{C_m \times \ZZ/p^{\alpha-1}} \\
\QQ_p
}
$$
Note that $n = [M: \QQ_p]$, so
$M$ embeds in $D$ as a maximal subfield.  Fix
such an embedding.  As in Section~\ref{sec:explicit}, there is an 
associated short exact sequence
\begin{equation}\label{eq:Eext}
1 \rightarrow M^\times \xrightarrow{i} E \xrightarrow{j} \Gal(M/\QQ_p)
\rightarrow 1 
\end{equation}
where $E$ is the normalizer of $M$ in $D^\times$.
The cohomology class $[E]$ of the extension (\ref{eq:Eext}) corresponds to the
Brauer group class of $D$ under the map
$$ H^2(M/\QQ_p)
\rightarrow \Br(\QQ_p) $$
and hence we have
$$ \Inv([E]) = \frac{1}{(p-1)p^{\alpha-1}m} \in \QQ/\ZZ. $$
Since we have $[L:\QQ_p] = p^{\alpha-1}m$, we deduce
that the image $[E']$ of $[E]$ under the map
$$ H^2(M/\QQ_p) \rightarrow
H^2(M/L) $$
has invariant
$$ \Inv([E']) = \frac{1}{p-1}. $$
The cohomology class $[E']$ classifies the extension $E'$ given by the
pullback:
$$
\xymatrix{
1 \ar[r] & 
M^\times \ar[r]^{i'} \ar@{=}[d] &
E' \ar[r]^-{j'} \ar[d] & 
\Gal(M/L) \ar[r] \ar@{^{(}->}[d] &
1 
\\
1 \ar[r] &
M^\times \ar[r]^i &
E \ar[r]^-{j} &
\Gal(M/\QQ_p) \ar[r] & 
1
}
$$
Let $\omega \in \QQ_{p^m} \subset L$ be a primitive $(p^m-1)$st root of unity.

\begin{lem}
The image of $\omega$ under the local Artin map
$$ \Art_L : L^\times \rightarrow
\Gal(M/L) = C_{p-1} $$
is a generator.
\end{lem}

\begin{proof}
The following diagram summarizes the relationship of $\Art_L$ to
$\Art_{\QQ_p}$.
$$
\xymatrix@C+2em@R+1em{
\mu_{p^m-1} \ar@{^{(}->}[r] \ar[d]^{\cong}_{(-)^{p^{\alpha-1}}} &
L^\times \ar[r]^-{\Art_L} \ar[d]_{N_{L/\QQ_{p^m}}} &
\Gal(M/L) \ar[d] \ar@{=}[r] & 
C_{p-1} \ar@{^{(}->}[d]
\\
\mu_{p^m-1} \ar@{^{(}->}[r] \ar@{->>}[d]_{(-)^{\frac{p^m-1}{p-1}}} &
\QQ_{p^m}^\times \ar[r]^-{\Art_{\QQ_{p^m}}} \ar[d]_{N_{\QQ_{p^m}/\QQ_p}} &
\Gal(M/\QQ_{p^m}) \ar[d] \ar@{=}[r] & 
(\ZZ/p^{\alpha})^\times \ar@{^{(}->}[d]
\\
\mu_{p-1} \ar@{^{(}->}[r]  &
\QQ_{p}^\times \ar[r]^-{\Art_{\QQ_{p}}} &
\Gal(M/\QQ_{p}) \ar@{=}[r] & 
C_m \times (\ZZ/p^{\alpha})^\times 
}
$$
The lemma follows from the fact that
$$ \Art_{\QQ_p}(\mu_{p-1}) = \Gal(M/L) \subset \Gal(M/\QQ_p). $$
\end{proof}

Using Proposition~\ref{prop:Ba}, we deduce the following.

\begin{cor}
The group $E'$ contains an element $b$ such that $j'(b)$ generates
$\Gal(M/L)$, and $b^{p-1} = i'(\omega)$.
\end{cor}

We deduce that there is a map of short exact sequences
$$
\xymatrix{
1 \ar[r] &
C_{p^{\alpha}(p^m-1)} \ar[r] \ar[d] &
G_\alpha \ar[r] \ar[d] &
C_{p-1} \ar[r] \ar[d]^{\cong} &
1
\\
1 \ar[r] &
M^\times \ar[r]_{i'} &
E' \ar[r]_-{j'} &
\Gal(M/L) \ar[r] &
1
}
$$
Since $E'$ is a subgroup of $E$, which in turn is a subgroup of $D^\times$,
we have given an embedding of $G_\alpha$ into $D^\times$.

\section{Negative results}\label{sec:negative}

Let $p$ be odd, and $n = (p-1)p^{\alpha-1}m$.
In this section we will make the following observation, which shows
that there are only a limited number of cases where all of $G_\alpha$
can be realized as an automorphism group of a height $n$ mod $p$ point of a
unitary Shimura stack of type $(1,n-1)$.

\begin{prop}\label{prop:negative}
Suppose that there is a choice of $F$, $GU$, and $K$ such that
$\Gamma(K)$ is isomorphic to $G_\alpha$.  Then $p \le 7$ and $m = 1$.
\end{prop}

\begin{proof}
Suppose that $G_\alpha = \Gamma(K)$ for some $F$, $GU$, and $K$.  Then
in particular $G_\alpha$ is a subgroup of the units of the algebra $D$ of
Section~\ref{sec:taf} of dimension $n^2$ over $F$.
There is an element $a \in G_\alpha$ which has order
$p^\alpha(p^m-1)$.  Therefore there is an embedding
$$ M = F(\zeta_{p^{\alpha}(p^m-1)}) \hookrightarrow D^\times. $$
Depending on $F$, we have
$$ [M:F] \in
\{(p-1)p^{\alpha-1}\varphi(p^m-1),
(p-1)p^{\alpha-1}\varphi(p^m-1)/2 \}. $$
Here, $\varphi$ is Euler's phi function.  
Since $[D:F] = n^2$, this is only possible if
$$ \varphi(p^m-1) \in \{ m, 2m \}. $$
It is easily checked, using the inequality $\phi(n) \ge \sqrt{n/2}$
and checking a few cases, that this can only happen if $p \in \{ 3,5,7\}$ and
$m = 1$.
\end{proof}

\section{Constructions of global division algebras}\label{sec:global}

Having established our negative results, we now begin the work necessary to  prove our positive results.  We wish to construct explicit global division algebras which contain Hewett's subgroups in their groups of units.  The tool we will use is the following global analog of Proposition~\ref{prop:Ba}.

\begin{prop}\label{prop:GlobalBa}
Let $K$ be a finite extension of $\QQ$, and suppose that 
$M$ is a degree $n$ cyclic extension of $K$.
Fix an injection
$$ \chi: \Gal(M/K) \hookrightarrow \QQ/\ZZ, $$
and let $\sigma \in \Gal(M/K)$ be the unique element satisfying
$$ \chi(\sigma) = 1/n. $$
Then for each element $a \in K^\times$ there is an extension
$$ 1 \rightarrow M^\times \xrightarrow{i} E_{a} \xrightarrow{j} \Gal(M/K)
\rightarrow 1
$$
such that 
\begin{enumerate}
\item there is a lift of $\sigma$ to $\td{\sigma} \in E$ such that
$\td{\sigma}^n = i(a)$, and
\item the corresponding central simple algebra $B_{[a]}$ has invariant
$$ \Inv_v B_{[a]} = \chi(\Art_{M/K}(a_v)) $$
at every place $v$ of $K$.
\end{enumerate}
\end{prop}

\begin{proof}
The extension $E_a$ may simply be defined by the presentation
$$ E_a = M^\times\bra{\td{\sigma}}/\bra{\td{\sigma}^n = a, \td{\sigma}x =
x^\sigma \td{\sigma}, x \in M^\times}. $$
We just need to compute the local invariants of the associated global
division algebra.
For a place $v$ of $K$, we have
$$ (B_{[a]})_v = B_{[a_v]}, $$
thus we have $\Inv_v B_{[a_v]} = \Inv B_{[a_v]}$.  
Let $v'$ be a place of $M$ dividing $v$.  We make use of the compatibility of the local and global Artin maps, given
by the following diagram.
$$
\xymatrix@C+1em{
K_v^\times \ar[r]^-{\Art_{M_{v'}/K_v}} \ar@{^{(}->}[d] &
\Gal(M_{v'}/K_v) \ar@{^{(}->}[d] 
\\
\mb{I}_K \ar[r]_-{\Art_{M/K}} &
\Gal(M/K)
}
$$
Let $k$ be the number of places of $M$ dividing $v$.  Then
$$ [\Gal(M/K): \Gal(M_{v'}/K_v)] = k. $$
The element $\sigma^k$ generates $\Gal(M_{v'}/K_{v}) \cong C_m$, where $m =
n/k$.  We
have $\chi(\sigma^{k}) = 1/m$.  Since we have
$$ M_v \cong \Ind_{C_{m}}^{C_n} M_{v'} \cong M_{v'}^k, $$
we may apply Proposition~\ref{prop:Ba} (Remark~\ref{rmk:Ba} if $v$ is
archimedean) together with the discussion following Example~\ref{ex:D1n} to
deduce that
$$ \Inv B_{[a_v]} = \chi(\Art_{M_{v'}/K_v}(a_v)) =
\chi(\Art_{M/K}(a)). $$
\end{proof}

Fix an odd prime $p$.  Let $n = (p-1)p^{\alpha-1}m$.
If $p = 3$ we will assume that $m > 1$.  The case of $p = 3$ and $m = 1$
will be treated separately.
In this section we show that there exist some natural choices of
global division algebras $D$ with center $F$, a global CM field, 
in which $G_\alpha$ embeds, such that there
exists a place $y\vert p$ of $F$ with $\Inv_y D = 1/n$.

Let $\omega$ be a primitive $(p^m-1)$st root of unity.  
Consider the tower of field extensions
$$
\xymatrix@C-1em{
\QQ(\omega) \ar@{-}[d]_{C_{m}}  \ar@{-} 	`/0pt[r]
					`/0pt[rdd]^{(\ZZ/(p^m-1))^\times}
					[dd]
& 
\\
F \ar@{-}[d]
\\
\QQ & 
}
$$
Here the Galois group $\Gal(\QQ(\omega),\QQ)$ is identified by
the isomorphism
\begin{align*} 
(\ZZ/(p^m-1))^\times & \cong \Gal(\QQ(\omega)/\QQ), \\
s & \mapsto ([s]: \omega \mapsto \omega^s),
\end{align*}
the
subgroup $C_m \le (\ZZ/(p^m-1))^\times$ is the subgroup generated
by the element $p$ mod $p^m-1$, and the field $F$ is the subfield of 
$\QQ(\omega)$ fixed by $C_m$.  The key property of $F$ is that $p$
splits completely in it.  The field $F$ is a CM field.  This is
because $\QQ(\omega)$ is a CM field, with conjugation given by the 
element $[-1] \in (\ZZ/(p^m-1))^\times$,
and for $p > 3$, or $p = 3$ and $m > 1$, the conjugation element $[-1]$ is not
contained in the subgroup $C_m$ generated by $p$.  We will write
$$ [-1](x) = \bar{x} $$
for $x \in \QQ(\omega)$.

Let $F^+$ be the totally real subfield of $F$ given by taking $[-1]$-fixed
points.  Let $\{x_i \}$ be the set of primes of $F^+$ dividing $p$.  Each
prime $x_i$ splits as $y_i \bar{y}_i$ in $F$.  The extension 
$\QQ(\omega)/F$ is totally unramified at the primes $y_i$ and
$\bar{y}_i$.  
Let $(t_i)$ be a finite sequence of integers coprime to
$p^m-1$, with $t_1 = 1$, 
such that the images
$[t_i]$
in $(\ZZ/(p^m-1))^\times \cong \Gal(\QQ(\omega)/\QQ)$ satisfy
$$ [t_i](y_1) = y_i. $$

\begin{thm}\label{thm:globalembed}
Let $D$ be a central
division algebra over $F$ whose non-trivial local invariants satisfy:
\begin{align*}
p^{\alpha-1} m \Inv_{y_i} D & \cong t_i/(p-1) \in \QQ/\ZZ, \\
p^{\alpha-1} m \Inv_{\bar{y}_i} D & \cong -t_i/(p-1) \in \QQ/\ZZ.
\end{align*}
Then the group $G_\alpha$ embeds in $D^\times$ as a 
maximal finite subgroup.
\end{thm}

In order to prove Theorem~\ref{thm:globalembed}, we will need
to introduce some additional field extensions.
Let $\zeta$ be a primitive $(p^\alpha)$th
root of unity, and consider the following field diagram.
$$
\xymatrix@C-3em@R-1em{
M = \QQ(\omega, \zeta) {\phantom{M}} \ar@{-}[dr]^{C_{p-1}}
\ar@{-}[dd]_{(\ZZ/p^\alpha)^\times} 
\\
& {\phantom{L =}} \QQ(\omega)[p^{\alpha-1}] = L
\ar@{-}[dl]^{\ZZ/p^{\alpha-1}} 
\\
\QQ(\omega) \ar@{-}[d]_{C_m} \\
F \ar@{-}[d]_{2} \\
F^+ \ar@{-}[d] \\
\QQ
}
$$
The extension $M/\QQ(\omega)$ is totally ramified over each of
the primes $y_i$ and $\bar{y}_i$.
Let $\omega_{y_i}$ denote the image of $\omega$ under the
inclusion
$$ L^\times \hookrightarrow L^\times_{y_i} 
\hookrightarrow \mb{I}_{L}.
$$
Define
$$ \sigma = \Art_{M/L}(\omega_{y_1}) \in \Gal(M/L) \cong
C_{p-1}. $$

We shall need the following lemma.

\begin{lem}\label{lem:artyi}
We have
$$ \Art_{M/L}(\omega_{y_i}) = \sigma^{t_i}. $$
The element $\sigma$ (and hence each of the elements
$\sigma^{t_i}$) is a
generator of $\Gal(M/L)$.
\end{lem}

\begin{proof}
Consider the following diagram.
$$
\xymatrix{
\mb{I}_L \ar[rr]^{\Art_{M/L}} \ar[d]_{N_{L/\QQ(\omega)}} &&
\Gal(M/L) \ar[d] \ar@{=}[r] &
C_{p-1} \ar@{^{(}->}[d] &
\\
\mb{I}_{\QQ(\omega)} \ar[rr]^{\Art_{M/\QQ(\omega)}}
\ar[d]_{N_{\QQ(\omega)/F}} &&
\Gal(M/\QQ(\omega)) \ar@{=}[r] \ar[d] &
(\ZZ/p^\alpha)^\times \ar@{^{(}->}[d]
\\
\mb{I}_{F} \ar[d]_{N_{F/\QQ}} \ar[rr]^{\Art_{M/F}} &&
\Gal(M/F) \ar[d] \ar@{=}[r] &
(\ZZ/p^\alpha)^\times \times C_m \ar@{^{(}->}[d] 
\\
\mb{I}_{\QQ} \ar[rr]_{\Art_{M/\QQ}} &&
\Gal(M/\QQ) \ar@{=}[r] &
(\ZZ/p^\alpha)^\times \times (\ZZ/p^m-1)^\times
}
$$
Viewing $\sigma = \Art_{M/L} (\omega_{y_1})$ as an element of $\Gal(M/\QQ)$
under the inclusion
$$ \Gal(M/L) \hookrightarrow \Gal(M/\QQ), $$
it suffices to show that
$$ \Art_{M/\QQ} (N_{L/\QQ}(\omega_{y_i})) = \sigma^{t_i}. $$
Note that by the definition of $\sigma$, we have
$$ \Art_{M/\QQ} (N_{L/\QQ}(\omega_{y_1})) = \sigma. $$
We compute:
\begin{align*}
N_{L/\QQ}(\omega_{y_i}) 
& = N_{F/\QQ} N_{\QQ(\omega)/F} N_{L/\QQ(\omega)} (\omega_{y_i}) \\
& = N_{F/\QQ} N_{\QQ(\omega)/F} (\omega_{y_i}^{p^{\alpha-1}}) \\
& = N_{F/\QQ} \left( \omega_{y_i}^{\frac{p^{m-1}-1}{p-1}p^{\alpha-1}}
\right) \\
& = N_{F/\QQ}(\omega_{y_1})^{\frac{p^{m-1}-1}{p-1}p^{\alpha-1}t_i}.
\end{align*}
Here, the last norm computation comes from the fact that under the
composite
$$ \Gal(\QQ(\omega)/\QQ) \twoheadrightarrow \Gal(F/\QQ) \xrightarrow{\cong}
\mr{Emb}(F, \QQ_p) $$
the element $[t_i]$ maps $\omega$ to $\omega^{t_i}$.
We have shown that
$$ N_{L/\QQ}(\omega_{y_i}) = N_{L/\QQ}(\omega_{y_1})^{t_i}. $$
Now apply $\Art_{M/\QQ}$ to the equation above.
\end{proof}

\begin{proof}[Proof of Theorem~\ref{thm:globalembed}]
Since $G_\alpha$ is maximal in $D^\times_{y_1}$, if it embeds in
$D^\times$, it must be maximal in $D^\times$.  
We therefore just must show it embeds.

Since each prime $y_i$ does not split in $M$,
we have $[L_{y_i}:F_{y_i}] = p^{\alpha-1}m$,  and $[M_{y_i}:F_{y_i}] =
(p-1)p^{\alpha-1}m$, which implies that
we have
\begin{align*}
\Inv_{y_i} L \otimes_F D & = t_i/(p-1), \\
\Inv_{\bar{y}_i} L \otimes_F D & = -t_i/(p-1), \\
\Inv_{y_i} L \otimes_F D & = 0, \\
\Inv_{\bar{y}_i} L \otimes_F D & = 0.
\end{align*}
In particular, $M$ splits $D$, and therefore $M$ embeds into $D$ as a
maximal commutative subfield
\cite[Thm.~9.18]{Swan}.  
Since $M$ embeds into $D$, the element $[D] \in \Br(F)$ is in the image of
the map
$$ H^2(M/F) \hookrightarrow \Br(F), $$
and is therefore corresponds to an extension
$$ 1 \rightarrow M^\times \rightarrow E \rightarrow (\ZZ/p^\alpha)^\times
\times C_{p^m-1} \rightarrow 1. $$
We let $[E]$ denote the corresponding cohomology class in $H^2(M/F)$.
Using the diagram
$$
\xymatrix{
H^2(M/F) \ar@{^{(}->}[r] \ar[d]_{\Res} &
\Br(F) \ar[d]
\\
H^2(M/L) \ar@{^{(}->}[r] &
\Br(L)
}
$$
together with with the local invariants of $L \otimes_F D$ computed above,
we deduce that the cohomology class
$$ \Res [E] \in H^2(M/L) $$
has 
\begin{align*}
\Inv_{y_i} [E] & = t_i/(p-1), \\
\Inv_{\bar{y}_i} [E] & = -t_i/(p-1).
\end{align*}
The class $\Res [E]$ is represented by the pullback $E'$ in the following
diagram.
$$
\xymatrix{
1 \ar[r] & 
M^\times \ar[r]^{i'} \ar@{=}[d] &
E' \ar[r]^{j'} \ar[d] &
C_{p-1} \ar[r] \ar@{^{(}->}[d] &
1
\\
1 \ar[r] & 
M^\times \ar[r]_{i} &
E \ar[r]_-{j} &
(\ZZ/p^\alpha)^\times \times C_{p^m-1} \ar[r] &
1
}
$$
Let 
$$ \sigma = \Art_{M/L}(\omega_{y_1}) \in \Gal(M/L) \cong C_{p-1} $$
be the generator considered earlier.  Choose 
$$ \chi: \Gal(M/L) \hookrightarrow \QQ/\ZZ $$
so that $\chi(\sigma) = 1/(p-1)$.
We compute (using Lemma~\ref{lem:artyi})
$$
\chi(\Art_{L}(\omega_{y_i})) = t_i/(p-1).
$$
By Proposition~\ref{prop:GlobalBa}, the extension $E'$ contains an element
$T$ so that $j'(T) = \sigma$ and $T^{p-1} = i'(\omega)$.  We deduce that there
is a map of short exact sequences
$$
\xymatrix{
1 \ar[r] & 
C_{p^\alpha(p^m-1)} \ar[r] \ar[d] &  
G_\alpha \ar[r] \ar[d] & 
C_{p-1} \ar[r] \ar@{=}[d] & 
1
\\
1 \ar[r] & 
M^\times \ar[r] & 
E' \ar[r] & 
C_{p-1} \ar[r] & 
1
}
$$
which maps $a$ to $\omega\zeta$ and $b$ to $T$.
Since $E'$ is contained in $E$, which in turn is a subgroup of $D^\times$,
we have embedded $G_\alpha$ in $D^\times$.
\end{proof}

\begin{rmk}
If $p = 3$ and $m = 1$, then the proof of Theorem~\ref{thm:globalembed} goes
through verbatim with the following modifications.  Take $F$ to be
\emph{any} quadratic imaginary extension in which the prime $3$ splits.
Replace all occurrences of $\QQ(\omega)$ with $F$, and replace the element 
$\omega$ with $-1$.  In this case, $L = F[3^{\alpha-1}]$, and $M =
F(\zeta_{3^{\alpha}})$.
\end{rmk}

\section{Involutions}\label{sec:involutions}

Let $n = (p-1)p^\alpha m$, and let $F$, $L$, $M$, $\omega$, $\zeta$, $\sigma$, 
and $D$ be as in Section~\ref{sec:global}.  Let $m \mapsto \bar{m}$ denote
the conjugation on the CM field $M$.  This conjugation 
automorphism is characterized by
\begin{align*}
\bar{\omega} & = \omega^{-1}, \\
\bar{\zeta} & = \zeta^{-1}.
\end{align*}
Let $D'$ be the division algebra associated to the extension
$$ 1 \rightarrow M^\times \rightarrow E' \rightarrow C_{p-1} \rightarrow 1
$$
where $E' = E_\omega$, in the notation of Proposition~\ref{prop:GlobalBa}.
The algebra $D'$ is the subalgebra of $D$ which centralizes the subfield
$L$ under the fixed embedding $M \subseteq D$.
Explicitly, $D'$ admits the presentation
\begin{equation}\label{eq:D'}
D' = M\bra{T}/(T^{p-1} = \omega, Tx = x^\sigma T, x \in M).
\end{equation}
We will express elements of $D'$ as
$$ \sum_{i = 0}^{p-2} x_i T^i $$
for $x_i \in M$.

In this section we will show that $D'$ admits a positive involution of the
second kind so that the subgroup $G_\alpha$ is contained in the unitary
group associated to the involution.  The involution we are interested in is
given by the following lemma.

\begin{lem}\label{lem:involution}
There is a unique involution $\dag'$ on $D'$ satisfying:
\begin{enumerate}
\item $x^{\dag'} = \bar{x}, \qquad x \in M$,
\item $T^{\dag'} = T^{-1} = \omega^{-1}T^{p-2}$.
\end{enumerate}
\end{lem}

\begin{proof}
We just need to check that the relations in the presentation (\ref{eq:D'}) 
are compatible with ${\dag'}$.  We
check
$$ (Tx)^{\dag'} = (x^\sigma T)^{\dag'} = T^{\dag'} (x^{\sigma})^{\dag'} = \omega^{-1}
T^{p-2} \bar{x}^{\sigma} = \bar{x} \omega^{-1} T^{p-2} = x^{\dag'}
T^{\dag'} $$
and 
$$
(T^{p-1})^{\dag'} = (\omega^{-1} T^{p-2})^{p-1} = \omega^{p-2} = \bar{\omega}
= \omega^{\dag'}.
$$
\end{proof}

By definition, the involution $\dag'$ is an involution of the second kind.

\begin{lem}
The involution $\dag'$ is positive.
\end{lem}

\begin{proof}
Let $x = \sum_{i = 0}^{p-2} x_i T^i$ be a non-zero element of $D'$.  
We must show
that $\Tr_{D/L}(x x^{\dag'})$ is a positive real number under all of the
complex embeddings of $L$.  

For a general $y = \sum y_i T^i$ in $D'$, we consider the $M$-linear
transformation:
\begin{align*}
R_y : D' & \rightarrow D', \\
z & \mapsto z \cdot y.
\end{align*}
Giving $D'$ the $M$-basis $\{T^i\}_{i = 1}^{p-2}$, we may represent $R_y$
by the following matrix.
$$
R_y = 
\begin{bmatrix}
y_0 & \omega^{-1}y_{p-2}^\sigma & \omega^{-1} y_{p-3}^{\sigma^2} & 
\ldots & \omega^{-1} y_{1}^{\sigma^{p-2}} 
\\
y_1 & y_0^\sigma & \omega^{-1} y_{p-2}^{\sigma^2} & & \vdots 
\\
y_2 & y_1^\sigma & y_0^{\sigma^2} & & \vdots 
\\
\vdots & \vdots & & \ddots & \vdots
\\
y_{p-2} & y_{p-3}^{\sigma} & \hdots & \hdots & y_0^{\sigma^{p-2}}
\end{bmatrix}
$$
In particular, we find that
$$ \Tr_{D'/L} (y) = \Tr_{M/L}(y_0). $$
We compute
\begin{align*}
xx^{\dag'} & = (x_0 + x_1 T + \cdots + x_{p-2} T^{p-2})(\bar{x}_0 +
\omega^{-1}T^{p-2}\bar{x}_1 + \cdots + \omega^{-1}T \bar{x}_{p-2}) \\
& = (x_0 \bar{x}_0 + x_1\bar{x}_1 + \cdots + x_{p-2} \bar{x}_{p-2}) +
\text{terms involving $T$.}
\end{align*}
We deduce that
$$ \Tr_{D'/L}(xx^{\dag'}) = \Tr_{M/L}(x_0 \bar{x}_0 + x_1\bar{x}_1 + \cdots +
x_{p-2} \bar{x}_{p-2}), $$
under each complex embedding of $L$, is a sum of positive real numbers,
hence positive.
\end{proof}

For a central simple algebra $B$ over a CM field $K$ with involution $*$ of
the second kind, we shall denote the associated unitary and unitary
similitude groups $U_{(B,*)}$ and $GU_{(B,*)}$.  These algebraic groups
(over $\QQ$) are given explicitly by
\begin{gather*}
U_{(B,*)}(\QQ) = \{ (B \times_\QQ R)^\times \: : \: 
x^*x = 1 \}, \\
GU_{(B,*)}(\QQ) = \{ (B \times_\QQ R)^\times \: : \: 
x^*x = R^\times \}, \\
\end{gather*}
for every $\QQ$-algebra $R$.

\begin{lem}\label{lem:unitary}
For each element $g \in G_\alpha \subset (D')^\times$, we have
$$ g \in U_{(D', \dag')}(\QQ). $$
\end{lem}

\begin{proof}
We just need to check on generators.  The subgroup $G_\alpha$, as
constructed in Theorem~\ref{thm:globalembed}, is generated by $T$ and
$\omega$.
We have
$$ T^\dag T = \omega^{-1} T^{p-2} T = \omega^{-1} \omega = 1 $$
and 
$$ \omega^\dag \omega = \bar{\omega} \omega = \omega^{-1} \omega = 1. $$
\end{proof}

\begin{prop}\label{prop:dag}
There exists a positive involution of the second kind $\dag$ on $D$ which
extends the involution $\dag'$ on $D'$. For any such extension $\dag$, we
have 
$$ G_\alpha \subseteq U_{(D,\dag)}(\QQ).
$$
\end{prop}

\begin{proof}
It suffices to show that the involution $\dag'$ on $D'$ extends as a
positive involution of the second kind to $D$.  This is established by
duplicating the argument of \cite[Lem.~9.2]{Kottwitz}.  The second
statement is immediate from Lemma~\ref{lem:unitary}.
\end{proof}

\section{Positive results}\label{sec:shaut}

Let $p$ be an odd prime.
Proposition~\ref{prop:negative} indicates that if $G_\alpha$ is going to be
the automorphism group of a mod $p$ point of the height $n$ locus of one of the
unitary Shimura stacks of type $(1,n-1)$ under consideration, with $n =
(p-1)p^{\alpha-1}m$, then $p \le 7$ and $m = 1$.  

In this section we will assume $p \in \{3, 5, 7\}$, and that $n = (p-1)p^{r-1}$,
and show that the finite group $G_r$ is the automorphism group of such
a mod $p$ point of a Shimura stack.  

For $p \in \{ 5, 7 \}$, let $F = \QQ(\omega)$, where, as in
Section~\ref{sec:global}, $\omega$ is a $(p-1)$st root of unity.  For $p =
3$, we can choose $F$ more freely: for concreteness choose $F =
\QQ(\sqrt{-2})$ in this case.  
Note that for
these primes, $F$ is a quadratic imaginary extension of $\QQ$ which splits
at $p$.
In each of these cases, $F = \QQ(\delta)$
where, for $p = 3,5,7$ we have $\delta = \sqrt{-2}, \sqrt{-1}, \sqrt{-3}$,
respectively.

We fix the rest of our defining Shimura data as follows:
\begin{align*}
B & = M_n(F), \\
* & = \text{conjugate transpose involution on $B$,} \\
\mc{O}_{B,(p)} & = M_n(\mc{O}_{F,(p)}), \\
V & = B, \\
\bra{x, y} & = \Tr_{F/\QQ} \Tr_{B/F} (x \beta y^*), \: 
\text{ for } \: \beta = 
\begin{bmatrix}
2\delta &   &   & 
\\
  & -2\delta&   & 
\\
  &   & \ddots &
\\
  &   &   & -2\delta
\end{bmatrix} \in B.
\end{align*}
Thus, the pairing $\bra{-,-}$ is the alternating hermitian form associated
(see \cite[Lem.~5.1.2]{taf})
to the symmetric hermitian form on $F^n$ given by the matrix
$$
\xi = 
\begin{bmatrix}
1 &   &   & 
\\
  & -1 &   & 
\\
  &   & \ddots &
\\
  &   &   & -1
\end{bmatrix}.
$$
We denote
$$ C = B^{op} = \End_B(V) $$
and let $\iota$ be the involution on $C$ induced from $\bra{-,-}$.
Let $GU$ be the associated unitary similitude group over $\QQ$, 
with $R$-points
\begin{align*}
GU(R)
= & \{ g \in C \otimes_\QQ R \: : \: g^\iota g \in R^\times \}.
\end{align*}

The goal of this section is to prove the following theorem.

\begin{thm}\label{thm:shaut}
For $p \in \{3, 5, 7\}$, and the group $GU$ described above, there exists a 
maximal compact open subgroup
$$ K \subset GU(\AF^{p,\infty}) $$
such that in the associated Shimura stack $\Sh(K)$, there exists a point
$$ \mathbf{A}_0 = (A_0, i_0, \lambda_0, [\eta_0]_K) \in
\Sh(K)^{[n]}(\bar{\FF}_p) $$
with
$$ \Aut(\mbf{A}_0) \cong G_r. $$
Under the map
$$ \Aut(\mbf{A}_0) \rightarrow \Aut(\epsilon\mbf{A}_0(u)) \cong \MS_n $$
the group $G_r$ embeds as a maximal finite subgroup with maximal $p$-order.
\end{thm}

\begin{rmk}
Theorem~\ref{thm:shaut} merely asserts that there exists a point with the desired automorphism group.  It is not, in general, the case that every mod $p$ point of the height $n$ locus has isomorphic automorphism group.  The desired mod $p$ point is globally determined by the (rather arbitrary) local choices of the maximal compact subgroups $K_\ell$ of (\ref{eq:Kell}). 
\end{rmk}

In order to prove Theorem~\ref{thm:shaut}, we will need several lemmas.
Let $K_0$ be any compact open subgroup of $GU(\AF^{p,\infty})$, and choose
an $\bar{\FF}_p$-point
$$ \mbf{A}_0 = (A_0, i_0, \lambda_0, [\eta_0]_{K_0}) 
\in \Sh(K_0)^{[n]}(\bar{\FF}_p).
$$
(By \cite[Prop.~14.3.2]{taf}, the set
$\Sh(K_0)^{[n]}(\bar{\FF}_p)$ is non-empty.)  Let $(D,\dag)$ be the
corresponding quasi-endomorphism ring with Rosati involution, as in
Section~\ref{sec:taf}.

Let $\zeta$, $M$, and $L$ be as in Section~\ref{sec:global}.  Let $L^+$ be
the fixed field under the CM-involution of $L$.  Note that in
the notation of Section~\ref{sec:global}, $F = \QQ(\omega)$ ($F =
\QQ(\sqrt{-2})$ if $p = 3$), $m = 1$, and
$\alpha = r$.  Furthermore, $L = F[p^{r-1}]$ and $L^+ = \QQ[p^{r-1}]$.
Note that the division algebra $D$ satisfies the hypotheses
of Theorem~\ref{thm:globalembed}, and therefore there exists an embedding
of $G_r$ in $D$.  
We need to show that an embedding exists so that there is
a containment
$$ G_r \subset GU_{\mbf{A}_0}(\ZZ_{(p)}), $$
and that there exists a compact open subgroup $K \subset GU(\AF^{p,\infty})$, 
so that there is a containment
$$ G_r \subset \Gamma(K) = GU_{\mbf{A}_0}(\ZZ_{(p)}) \cap K = \Aut (A_0,
i_0, \lambda_0, [\eta_0]_{K}). $$

We pause to recollect 
the essential facts concerning the classification of involutions
of the second kind
on a central simple algebra $E$ over $K$, a quadratic extension of $K_0$,
a local or global number field (see, for instance, \cite{Scharlau}, or
\cite[Ch.~5]{taf}, though the latter only treats the case where 
$K_0 = \QQ$ or $\QQ_v$).  
Suppose that $\dim_K E = d^2$. 
Fix a positive involution $*$ on $E$ of the second kind.
There are
bijections
\begin{align*}
H^1(K_0, GU_{(E,*)}) & \cong 
\left\{ 
\begin{array}{l}
\text{Similitude classes of non-degenerate} \\ 
\text{$*$-hermitian $*$-symmetric forms:} \\
\quad E \otimes E \rightarrow K
\end{array}
\right\}
\\
& \cong
\left\{ 
\begin{array}{l}
\text{Equivalence classes of involutions} \\ 
\text{of the second kind on $E$} 
\end{array}
\right\}.
\end{align*}

In the case where $K$ is a global field,
if $d$ is even, the group $GU_{(E,*)}$ satisfies the Hasse
principle \cite[Sec.~7]{Kottwitz}: the map
$$ H^1(K_0, GU_{(E,*)}) \rightarrow \prod_v H^1((K_0)_v, GU_{(E,*)}) $$
is an injection, and an involution $(-)^\#$ of the second kind on $E$ is
determined by the involution it induces on $E_v$ for each place $v$.  Using
the Noether-Skolem theorem and Hilbert Theorem 90 to express an involution
$(-)^\#$ on $E$ by
$$ x^\# = \xi^{-1} x^* \xi $$
for $\xi \in E$ with $\xi^* = \xi$, the associated hermitian form is given
by
$$ (x,y) = \Tr_{E/K}(x \xi y^*). $$ 
More precisely, the involution is determined by the
discriminant of the associated hermitian form
$$ \disc_{(-,-)} := N_{E/F}(\xi) \in K_0^\times/N(K^\times) $$
and the (unordered) signature of 
the form $(-,-)$ at each of the real places of $K_0$
which do not split in $K$.  The local Galois cohomology $H^1((K_0)_v,
GU_{(E,*)})$
is only nontrivial for places $v$ of $K_0$ which do not split in $K$.

Assume now that $K$ is a non-archimedean local number field, and that $E
\cong
M_d(K)$.  If $d > 2$,
the existence of the involution $*$ actually forces $E$ to be split.
Assume that $*$ is the involution on $M_d(K)$ given by 
conjugate transpose.  The associated hermitian form $(-,-)$ may be regarded
as the hermitian form on $K^d$, associated to the matrix $\xi$. The
isometry class of $(-,-)$ is classified by its
discriminant
$$ \disc_{(-,-)} := \det \xi \in K_0^\times/N(K^\times) \cong \ZZ/2. $$
If $d$ is odd, any two hermitian forms lie in the same similitude class
(see \cite[Cor.~3.5.4]{taf}).  If $d$ is even, then if two forms are in the
same similitude class, they are isometric.  Therefore, the invariant
$\disc$ is an invariant of the involution associated to the hermitian form.
When $d$ is even, the similitude class of the form is 
determined by its Witt index, the
dimension of a maximal totally isotropic subspace of $K^d$ for $(-,-)$.  We
have \cite[Sec.~12.3]{taf}
\begin{align*}
\disc_{(-,-)} = (-1)^{d/2} \in K_0^\times/N(K^\times) & \Leftrightarrow
\text{Witt index} = d/2, \\
\disc_{(-,-)} \ne (-1)^{d/2} \in K_0^\times/N(K^\times) & \Leftrightarrow
\text{Witt index} = d/2-1. \\
\end{align*}
If the form arises from the completion of a global form, the
local discriminant is simply the image of the global discriminant under
the completion map, \emph{provided the global involution $*$ is equivalent
to conjugate transpose after completion}. 

The following lemmas give some useful local presentations of the
conjugate transpose involution on a local matrix algebra.  We will use
these lemmas later in this section to compute the local invariants of the
involution $\dag'$ constructed in Lemma~\ref{lem:involution}.

\begin{lem}\label{lem:nonsplit}
Let $K/K_0$ be an unramified quadratic extension of local nonarchimedean
number fields.    
Suppose that $K'/K$ is the unramified extension of degree
$n$, so that $\Gal(K'/K) = C_n = \bra{\sigma}$.  Let $\br{(-)}$ denote the
unique element of $\Gal(K'/K_0) \cong C_{2n}$ of order $2$.
Let $E$ be the (split) central simple algebra over $K$ 
associated to the split exact sequence
$$ 1 \rightarrow (K')^\times \rightarrow (K')^\times \rtimes C_n
\rightarrow C_n \rightarrow 1. $$
Let $\#$ be the unique involution on $E$ such that
\begin{align*}
x^\# & = \bar{x} \qquad \text{for $x \in K'$}, \\
\sigma^\# & = \sigma^{-1}.
\end{align*}
Then the pair $(E, \#)$ is isomorphic to $M_n(K)$ with the
conjugate-transpose involution.
\end{lem}

\begin{proof}
Observe that there is a natural way to identify $E$ with $\End_K(K')$:
embed $K'$ in $\End_K(K')$ via its action by multiplication, and regard
$\sigma$ as a $K$-linear endomorphism of $K'$.  Consider the hermitian form
\begin{align*}
(-,-): K' \times K' & \rightarrow K, \\
(x,y) & \mapsto \Tr_{K'/K}(x\bar{y}).
\end{align*}
Observe that for $a, x, y \in K'$, we have
\begin{gather*}
(ax,y) = (x,\bar{a}y) \\
(\sigma(x), y) = (x, \sigma^{-1}(y)) 
\end{gather*}
which implies that the involution $\#$ on $E$ is precisely the involution
associated to the form $(-,-)$.  We need to show that $\disc_{(-,-)} \equiv
1 \in K_0^\times/N(K^\times)$.

Choose a basis $(e_1, \ldots, e_n)$ 
of $K'/K$ so that the hermitian form $(-,-)$ is represented by a
matrix of the form
$$
\xi = 
\begin{bmatrix}
a_1 & & \\
& \ddots & \\
& & a_n
\end{bmatrix}
$$
for $a_i \in K_0^\times$.  Consider the symmetric bilinear form
$$ \Tr_{K/K_0}(-,-): K' \times K' \rightarrow K_0. $$
Write $K = K_0(\delta)$ for $\delta^2 = -d \in K_0$.  With respect to the
basis
$$ (e_1, \delta e_1, e_2, \delta e_2, \ldots, e_n, \delta e_n) $$
the symmetric bilinear form $\Tr_{K/K_0}(-,-)$ is represented by the
following matrix.
$$
\xi_0 = 
\begin{bmatrix}
2a_1 & & & &  \\
& 2da_1 & & &  \\
 & & \ddots & &  \\
 & & & 2a_n & \\
 & & & & 2da_n \\
\end{bmatrix}
$$
We deduce that
$$
\det \xi_0 = (4d)^n \det \xi.
$$
Since $K/K_0$ is unramified, $d \equiv 1$ in $K_0^\times/N(K^\times)$.
Since $4 \in N(K^\times) \subset K_0^\times$, we deduce that
$$ \disc_{(-,-)} = \det \xi \equiv \det \xi_0 \in K_0^\times/N(K^\times). $$
Now, the trace pairing
\begin{align*}
(-,-)_{K'/K_0}: K' \times K' & \rightarrow K_0, \\
(x,y) & \mapsto \Tr_{K'/K_0}(xy)
\end{align*}
is represented by the matrix
$ \xi_0 \cdot c$
where $c$ is the $K_0$-linear endomorphism of $K'$ corresponding to
$\br{(-)}$.  Note that since $c^2 = \mr{Id}$, we have $\det c \in \{ \pm 1
\}$.
Since $K/K_0$ is unramified, Corollary~1 of Section~III.5 of
\cite{Serrelocal} implies that 
$$ \det (\xi_0 \cdot c) \not\equiv \pi \in K_0^\times/(K_0^\times)^2, $$
where $\pi$ is a uniformizer of $K_0$.  Since $K/K_0$ is unramified, 
$\pi$ generates the group 
$K_0^\times/N(K^\times)$ and we deduce that
$$ \disc_{(-,-)} \equiv \det (\xi \cdot c) \equiv 1 \in
K_0^\times/N(K^\times), $$
as desired.
\end{proof}

\begin{lem}\label{lem:split}
Let $K_0$, $K$, $K'$, $n$, $\br{(-)}$, and $\sigma$ be as in
Lemma~\ref{lem:nonsplit}.  Let $E_m$ be the (split) central 
simple algebra over $K$ 
associated to the split exact sequence
$$ 1 \rightarrow ((K')^\times)^m \rightarrow ((K')^\times)^m \rtimes C_{nm}
\rightarrow C_{nm} \rightarrow 1, $$
where $C_{nm}$ acts on $((K')^\times)^m$ through the isomorphism
$$ ((K')^\times)^m \cong \Ind^{C_{mn}}_{C_n} (K')^\times, $$
as in the discussion following Example~\ref{ex:D1n}.  Let $S$ be an 
element of $E_m$ corresponding to a generator of $C_{nm}$.
Let $\#_m$ be the unique involution on $E_m$ such that
\begin{align*}
(x_{1}, \ldots , x_{m})^{\#_m}
& = (\bar{x}_{1}, \ldots, \bar{x}_{m}) 
\qquad \text{for $(x_i) \in (K')^m$}, \\
S^{\#_m} & = S^{-1}.
\end{align*}
Then the pair $(E_m, \#_m)$ is isomorphic to $M_{nm}(K)$ with the
conjugate-transpose involution.
\end{lem}

\begin{proof}
Note that the case of $m = 1$ is precisely the content of
Lemma~\ref{lem:nonsplit}.
Using Lemma~\ref{lem:nonsplit}, fix an isomorphism
$$ (E_1, \#_1) \cong (M_{n}(K), (-)^*), $$
where $(-)^*$ denotes conjugate-transpose.  Under this isomorphism, we may
represent elements of $K'$, as well as $\sigma$, as giving $n\times n$
matrices.  Consider the diagonal embedding
\begin{align*}
((K')^\times)^m & \hookrightarrow M_{nm}(K)^\times, \\
(x_1, \ldots, x_m) & \mapsto 
\begin{bmatrix}
x_1 & & \\
& \ddots & \\
& & x_m 
\end{bmatrix}
\end{align*}
where here, and elsewhere in this proof, $nm \times nm$ matrices are
represented by $m \times m$ matrices of blocks, with each block
corresponding to an $n \times n$ matrix.
Under this embedding, 
the conjugate transpose involution $(-)^*$ on $M_{nm}(K)$ restricts to
$$ (x_1, \ldots, x_m) \mapsto (\bar{x}_1, \ldots , \bar{x}_m ). $$
We may extend this embedding to an embedding 
$$ ((K')^\times)^m \rtimes
C_{nm}  \hookrightarrow (M_{nm}(K))^\times $$
by sending $S$ to the matrix
$$
\begin{bmatrix}
0 & 1 & & & \\
& 0 & 1 & & & \\
& & \ddots & \ddots & \\
& & & & 1 \\
\sigma & & & & 0 
\end{bmatrix}.
$$
It is easily checked that the transpose of this matrix is the matrix
corresponding to $S^{nm-1}$.
\end{proof}

We now determine the invariants of our involution $\iota$ on $C \cong
M_n(F)$ which defines the group $GU$.  

\begin{lem}\label{lem:iotainvs}
For each prime $\ell$ which does not split in $F$, the Witt index of
$GU(\QQ_\ell)$ is $n/2$, except in the case of $p = 5$ and $\ell = 2$,
for which the Witt index is $n/2-1$.
\end{lem}

\begin{proof}
The discriminant of $\iota$ is $(-1)^{n-1}$.  Since
$$ n = (p-1)p^{r-1}, $$
and $p$ is odd, $n-1$ is odd, and $\disc = -1$.
If $\ell$ is unramified, then
$\disc$ is zero in $\QQ_\ell/N(F_\ell^\times)$.  If $p = 3,7$, the quantity
$(p-1)/2$ is odd, and hence $\disc = (-1)^{n/2}$.  However, if $p = 5$,
$n/2$ is even, and $\disc \ne (-1)^{n/2}$.
\end{proof}

\begin{lem}\label{lem:embedding}
There exists an embedding $M \hookrightarrow D$ so that $\dag$ restricts to
the conjugation on $M$.
\end{lem}

\begin{proof}
By \cite[Thm.~3.2]{PrasadRapinchuk2}, it suffices to prove that there are
local embeddings
$$ (M_v, \overline{(-)}) \hookrightarrow (D_v, \dag) $$
for each place $v$ of $\QQ$.  Since $M$ is a CM field, the case where $v$
is the infinite place
is easily verified.  If $v$ is a finite place which splits in $F$, the
local embedding follows from \cite[Prop.~A.3]{PrasadRapinchuk1}.  If $v$ is
finite and unramified in $F$, then the isomorphism
$$ (C_v, \iota) \cong (D_v, \dag) $$
induced from the level structure $\eta_0$ 
implies that, using Lemma~\ref{lem:iotainvs}, except in the
case of $p = 5$ and $\ell = 2$, the Witt index of $\dag$ is $n/2$.  The
local embeddability of $(M_v, \overline{(-)})$ then follows from
\cite[p.~340]{PlatonovRapinchuk}.  

We are left with the case of $p = 5$ and $\ell = 2$.
In this case, $F = \QQ(i)$ and the
Witt index of $\dag$ is $n/2-1$.  It suffices to prove that
$(M,\overline{(-)})$
embeds in $(M_n(F_2),\tau)$ for any involution $\tau$ of Witt index
$n/2-1$, since Witt index determines the equivalence class of an involution.  
Let $\tau_0$
be an involution of $M_n(F_2)$ with Witt index $n/2$, and use
\cite[p.~340]{PlatonovRapinchuk} to produce an embedding
$$ (M_2, \overline{(-)}) \hookrightarrow (M_n(F_2), \tau_0). $$
Since $n \equiv 0 \mod 4$, we deduce that 
$$ \disc_{\tau_0} = 1 \in \QQ_2^\times/N(F_2^\times) = (\ZZ/4)^\times. $$
Define, for $x \in M_n(F_2)$
$$ x^\tau = \xi^{-1}x^{\tau_0} \xi $$
for
$$ \xi = (1+i)\zeta + (1-i)\zeta^{-1} \in M_2 \subset M_n(F_2). $$
Note that since we have 
$$ \xi^{\tau_0} = \br{\xi} = \xi, $$
the transformation $(-)^{\tau}$ defines an involution of the second kind 
on $M_n(F_2)$.
Since $\xi \in M$, we easily see that $x^{\tau} = \br{x}$ for $x \in M$.
In particular, we have an embedding
$$ (M_2, \br{(-)}) \hookrightarrow (M_n(F_2), \tau). $$
Using the fact that $n \equiv 4 \mod 8$, together with $(i+1)^2 = 2i$, we
compute
\begin{align*}
\disc_{\tau} & = \det(\xi) \\
& = N_{M_2/F_2}((1+i)\zeta + (1-i)\zeta^{-1}) \\
& = -2^{n/2}.
\end{align*}
Since $2 = N_{F_2/\QQ_2}(1+i)$, we deduce that
$$ -2^{n/2} \equiv -1 \in \QQ_2^\times/N(F_2^\times) $$
and therefore that the Witt index of $\tau$ is $n/2-1$.  We have therefore
produced the desired local embedding.
\end{proof}

Fix an embedding
$$ (M, \br{(-)}) \hookrightarrow (D, \dag) $$
as in Lemma~\ref{lem:embedding}.  Let $D' \subset D$ be the subalgebra
which centralizes the subfield $L \subset M$.  Since $\dag$ restricts to
an involution of $M$, we deduce that $\dag$ restricts to $D'$.  Let $\dag'$
be the involution of $D'$ constructed in Lemma~\ref{lem:involution}.

\begin{lem}\label{lem:dagdag'}
The involution $\dag$ restricted to $D'$ is equivalent to the involution
$\dag'$.
\end{lem}

In order to prove Lemma~\ref{lem:dagdag'}, we shall need the following.

\begin{lem}\label{lem:dag'}
Let $v$ be a finite place of $L^+$ which is inert in $L$.  Then we
have
$$ \disc_v(\dag') \equiv 1 \in (L^+_v)^\times/N(L_v^\times). $$
\end{lem}

\begin{proof}
Let $(v_i)$
be the collection of places of $M$ which lie over $v$, so that $M_v$
splits into a product
$$ M_v = \prod_i M_{v_i} $$
of (isomorphic) field extensions of $L_v$.
Since $v$ is inert in $L$, $v$ cannot lie over $p$, and therefore each of
the extensions $M_{v_i}/L_v$ is unramified.  Let $d = [M_{v_i}:L_v]$.
The division algebra $D'_v$ corresponds to the extension
$$ 1 \rightarrow \prod_i M_{v_i}^\times \rightarrow E'_v \rightarrow C_{p-1} 
\rightarrow 1 $$
where there exists an element $T \in E'_v$ which maps to a generator
$\sigma$ of
$C_{p-1}$, and for which 
$$ T^{p-1} = ( \omega_{v_i}) \in \prod M_{v_i}^\times. $$
Note $\omega_{v_i}$ is contained in the subfield $L_v \subset M_{v_i}$, so
we will simply denote the corresponding element $\omega_v \in L_v$.
Since $M_{v_i}/L_v$ is inert, the element $\omega_{v}$ must lie in
$N(M_{v_i}^\times)$.  Let $\beta_1 \in M_{v_1}$ be chosen such that
$$ N_{M_{v_1}/L_v}(\beta_1) = \omega_{v_1}. $$
We then have
$$ N_{M_{\sigma^i{v_1}}/L_v}(\beta_1^{\sigma^i}) = \omega_{\sigma^i{v_1}}. $$
Set $\beta = (\beta_1, 1, \ldots, 1) \in \prod M_{v_i}^\times$.  Define
$$ S = \beta^{-1}T \in D'. $$
We compute
\begin{align*}
S^{p-1} & = \beta^{-1}T\beta^{-1}T \cdots \beta^{-1}T \\
& = \beta^{-1}\beta^{-\sigma} \cdots \beta^{-\sigma^{p-2}}T^{p-1} \\
& = \omega_v^{-1} \omega_v \\
& = 1.
\end{align*}
The effect of the involution $\dag'$ on $S$ is given by
\begin{align*}
S^{\dag'} & = (\beta^{-1}T)^{\dag'} \\
& = \omega_v^{-1}T^{p-2}\bar{\beta}^{-1} \\
& = S^{p-2} \beta^{-1} \bar{\beta}^{-1}.
\end{align*}
By Lemma~\ref{lem:split}, to finish the proof, we must show that the
element $\beta_1 \in M_{v_1}$ 
may be chosen such that $\bar{\beta}_1 = \beta_1^{-1}$.
Let $t$ and $\ell$ be such that $L_v = \QQ_{\ell^t}$.  The extension
$M_{v_1}$ must be isomorphic to $\QQ_{\ell^{dt}}$.  Since $\omega_v$ is a
root of unity, it must be contained in
$$ \mu_{\ell^t-1} \subset L_v^\times. $$
Under the norm map
$$ N_{M_{v_1}/L_v}: M_{v_1}^\times \rightarrow L^\times_v $$
the subgroup $\mu_{\ell^{dt}-1} \subset M_{v_1}^\times$ surjects onto the
subgroup $\mu_{\ell^t-1} \subset L_v^\times$.  Therefore, we may take
$\beta_1$ to be a root of unity in $M_{v_1}$.  For such a choice of
$\beta_1$, we have $\bar{\beta}_1 = \beta_1^{-1}$, as desired.
\end{proof}

\begin{proof}[Proof of Lemma~\ref{lem:dagdag'}]
Since $n$ is even, the group $GU_{(D',\dag')}$ satisfies the Hasse
principle \cite[Sec.~7]{Kottwitz}, and therefore it suffices to show that
the involution $\dag'$ has the same local invariants as the involution
$\dag$.  Since both involutions are positive, we just need to show that
the local invariants agree at the finite places of $L^+$.

Expressing $\dag'$ as 
$$ x^{\dag'} = \xi^{-1} x^\dag \xi $$
for $x \in D'$, where $\xi = \xi^\dag$, define
$$ \Delta = N_{D/L}(\xi) \in (L^+)^\times/N(L^\times). $$
It suffices to show that
$$ \Delta_v \equiv 1 \in (L^+_v)^\times/N(L^\times_{v}) $$
for every finite place $v$ of $L'$.  
Note that $F$ has
the property that there is precisely one prime $\ell$ of $\QQ$ which
ramifies in $F$ ($\ell = 2$ if $p \in \{3,5 \}$ and $\ell = 3$ if $p = 7$).  
Since
in each of these cases, $\ell$ is a generator of $\ZZ_p^\times$, we deduce
that the extension $L^+/\QQ$ is inert at $\ell$, and therefore $(\ell)$ is
prime in $L^+$, and is the unique prime of $L^+$ which ramifies in $L$.  
Since $\dag'$ is positive,
$\Delta$ is positive at every archimedean place of $L^+$.  In light of the
fundamental exact sequence
$$ 0 \rightarrow (L^+)^\times/N(L^\times) \rightarrow \bigoplus_v
(L^+_v)^\times/N(L_v^\times) \rightarrow \ZZ/2 \rightarrow 0, $$
we see that it suffices for us to verify that $\delta_v \equiv 1$ for only the
places $v$ of $L^+$ which are inert in $L$.  Let $v$ be such a place, and
let $\ell'$ be the rational prime lying under $v$.
We have the following field diagram of unramified extensions
$$ 
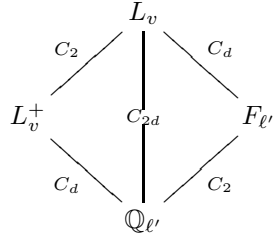
\xymatrix{
& L_v \ar@{-}[dl]_{C_2} \ar@{-}[dr]^{C_d} \ar@{-}[dd]|{C_{2d}} \\
L^+_v \ar@{-}[dr]_{C_d} && F_{\ell'} \ar@{-}[dl]^{C_2} \\
& \QQ_{\ell'}
} $$
where $d = [L^+_v:\QQ_\ell]$ divides $p^{r-1}$, and is therefore odd.
Under the embedding $GU_{(D', \dag)} \hookrightarrow GU_{(D,\dag)}, $
the local discriminants give the following commutative diagram.
$$
\xymatrix{
H^1(L^+_v, GU_{(D',\dag)}) \ar[r] \ar[d]_{\disc_v}^\cong &
H^1(\QQ_{\ell'}, GU_{(D, \dag)}) \ar[d]^{\disc_{\ell'}}_\cong \\
(L_v^+)^\times/N(L_v^\times) \ar[r]_{N_{L^+_v/\QQ_{\ell'}}} &
\QQ_{\ell'}^\times/N(F_{\ell'}^\times)
}
$$
Since all of the extensions are unramified, 
the compatibility of the local Artin maps allows us to verify that the map
$$ (L_v^+)^\times/N(L_v^\times) \xrightarrow{N_{L^+_v/\QQ_{\ell'}}}
\QQ_{\ell'}^\times/N(F_{\ell'}^\times)
$$
is an isomorphism.  The level structure $\eta_0$ gives an isomorphism
$$ GU_{(D,\dag)}(\QQ_{\ell'}) \cong GU(\QQ_{\ell'}). $$
Since $\ell'$ is unramified in $F$, Lemma~\ref{lem:iotainvs} implies that we
have 
$$ \disc_{\ell'}(\iota) = -1 \equiv 1 \in \QQ_{\ell'}^\times/N(F_{\ell'}^\times). $$
We deduce that
$$ \disc_v(\dag) \equiv 1 \in (L^+_v)^\times/N(L_v^\times). $$
By Lemma~\ref{lem:dag'}, we have 
$$ \disc_v(\dag') \equiv \disc_v(\dag). $$
We conclude that $\Delta_v \equiv 1$, as desired.
\end{proof}

\begin{proof}[Proof of Theorem~\ref{thm:shaut}]
By Lemma~\ref{lem:dagdag'}, the involution $\dag$ is equivalent to the
involution $\dag'$.  Thus we get an isomorphism $GU_{(D', \dag')} \cong
GU_{(D', \dag)}$.  By Lemma~\ref{lem:unitary}, $G_r$ embeds in $GU_{(D',
\dag')}(\QQ)$.  We may therefore have an embedding
$$ G_r \hookrightarrow GU_{(D',\dag')}(\QQ) \cong GU_{(D', \dag)}(\QQ)
\hookrightarrow GU_{(D, \dag)}(\QQ) = GU_{\mbf{A}_0}(\QQ). $$
Since we have
$$ GU_{\mbf{A}_0}(\QQ_p) \cong D_u^\times, $$
the group of units in a central division algebra of invariant $1/n$ over
$\QQ_p$, the group $GU_{\mbf{A}_0}(\ZZ_p)$ is the unique maximal compact
subgroup.  Since $G_r$ is compact, it must be contained in
$GU_{\mbf{A}_0}(\ZZ_p)$, and we deduce
$$ G_r \subset GU_{\mbf{A}_0}(\ZZ_p) \cap GU_{\mbf{A}_0}(\QQ) =
GU_{\mbf{A}_0}(\ZZ_{(p)}). $$
Using the level structure $\eta_0$ to give an
isomorphism
$$ GU_{\mbf{A}_0}(\AF^{p,\infty}) \cong GU(\AF^{p,\infty}) $$
we may regard $G_r$ as a subgroup of $GU(\AF^{p,\infty})$.
For each prime $\ell \ne p$, choose a maximal compact subgroup
\begin{equation}\label{eq:Kell} 
K_\ell \subset GU(\QQ_\ell) 
\end{equation}
which contains the image of $G_r$.  
Let
$$ K = \prod_\ell K_\ell \subset GU(\AF^{p,\infty}) $$
be the associated compact open subgroup.  We have
$$ G_r \subseteq GU_{\mbf{A}}(\ZZ_{(p)}) \cap K = \Gamma(K). $$
Under the completion map
$$ GU_{\mbf{A}_0}(\ZZ_{(p)}) \hookrightarrow GU_{\mbf{A}_0}(\ZZ_p) \cong
\MS_n $$
the finite group $\Gamma(K)$ embeds into $\MS_n$.  Since $G_r$ is a maximal
finite subgroup of $\MS_n$, we conclude that
$$ G_r = \Gamma(K). $$
However, as discussed in Section~\ref{sec:taf}, the group $\Gamma(K)$ is
the automorphism group of the point
$$ (A_0, i_0, \lambda_0, [\eta_0]_K ) \in \Sh(K)^{[n]}(\bar{\FF}_p). $$
The theorem is therefore proven.
\end{proof}

\section{Concluding remarks}\label{sec:rmks}

In this final section we give a brief discussion of the relationship of the
results in this paper to some topics in homotopy theory, as well as sketch
some potential extensions.

\subsection{Relationship to Hopkins-Gorbounov-Mahowald theory}

Let $p \ge 3$.
In \cite{GorbounovMahowald}, the cohomology theory
$$ EO_{p-1} = E_2^{hG_1} $$
is related to liftings to characteristic zero of certain curves in
characteristic $p$.  

To summarize their work, let $C/\FF_p$ be the curve
given by
$$ C: y^{p-1} = x^{p} - x. $$
The curve $C$ has genus $(p-1)(p-2)/2$.  Therefore the Jacobian $J(C)$ is
an abelian variety of dimension $(p-1)(p-2)/2$, and is also acted upon by
$G_1$.  The action of the subgroup $\mu_{p-1} < G_1$ induces a splitting of
the formal completion $\widehat{J}(C)$ into $p-2$ summands, according to
the weights of the action:
$$ \widehat{J}(C) = \widehat{J}(C)[1] \oplus \widehat{J}(C)[2] \oplus
\cdots \oplus \widehat{J}(C)[p-2]. $$
The dimensions of the summand $\widehat{J}(C)$ is $i$.  Each summand has
height $p-1$.  In particular, $\widehat{J}(C)[1]$ is $1$ dimensional and of
height $p-1$, and there is an induced embedding
$$ G_1 \hookrightarrow \Aut(\widehat{J}(C)[1] \otimes_{\FF_p} \bar{\FF}_{p}) 
\cong \MS_{p-1} $$
as a maximal finite subgroup.

The second author, with Gorbounov and Mahowald, 
constructed a lift $\td{C}$ of the curve
$C$ over the ring
$$ E = \ZZ_p[[u_1, \ldots, u_{p-2} ]] $$
such that the action of $G_1$ lifts to an action on $\td{C}$.  Here, the
group $G_1$ acts non-trivially on the ring $E$, but the subgroup
$\mu_{p-1}$ acts trivially.  The authors prove that the
deformation $\widehat{J}(\td{C})[1]$ of $\widehat{J}(C)[1]$ is a universal
deformation. This gives a $G_1$-equivariant isomorphism from $E$ to the
Lubin-Tate universal deformation ring of $\widehat{J}(C)[1]$.  This gives
explicit formulas for the action of $G_1$ on $\pi_0(E_{p-1})$.

We explain how this set-up relates to the results of this paper.  Assume
that $p \ge 5$.
Let $F = \QQ(\omega)$ where $\omega$ is a primitive $(p-1)$st root of
unity.  The action of $\mu_{p-1}$ on $\td{C}$ gives the
$(p-1)(p-2)/2$-dimensional abelian variety $J(C)$ an action (through
quasi-endomorphisms) by the ring $\QQ[z]/(z^{p-1} - 1)$.  Factorize
$$ z^{p-1}-1 = f_1(z)\cdots f_d(z) $$
as a product of irreducible polynomials, so that $f_1$ is of degree
$\phi(p-1)$.  Then there is a product decomposition
$$ \QQ(z)/(z^{p-1}-1) \cong F_1 \times \cdots \times F_d $$
where $F_1 \cong F$.  The abelian variety $J(C)$ is then quasi-isogenous to
a product
$$ J(C) \simeq J(C)_1 \times \cdots \times J(C)_d $$
where the factor $J(C)_i$ has complex multiplication by the field $F_i$.
In particular, $J_1(C)$ has complex multiplication by $F$: we get
$$ i_C: F \rightarrow D_C := \End^0(J(C)_1). $$
Taking the
formal completion, the summands of $\widehat{J}(C)$ which show up in
$\widehat{J}(C)_1$ correspond to the weights $i$ for $i \in
(\ZZ/(p-1))^\times$.  We therefore have a decomposition
$$ \widehat{J}(C)_1 = \bigoplus_{i \in (\ZZ/(p-1))^\times}
\widehat{J}(C)[i]. $$

Let $(t_1, \ldots, t_k)$ be a sequence of integers with $k = \phi(p-1)/2$ 
such that $0 < t_j < p-1$, $t_1 = 1$, and the sequence
$$ (t_1, \ldots, t_k, (p-1)-t_1, \ldots, (p-1)-t_k) $$
gives a complete list of the elements of $(\ZZ/(p-1))^\times$ when reduced
mod $(p-1)$.
The prime $p$ splits completely in $F$, and we write
$$ (p) = y_1 \ldots y_k \bar{y}_1 \ldots \bar{y}_k $$
where,
regarding $(\ZZ/(p-1))^\times = \Gal(F/\QQ)$, we have
\begin{align*}
[t_j](y_1) & = y_j \\
[-t_j](y_1) & = \bar{y}_j.
\end{align*}
The decomposition
$$ F_p = \prod_j F_{y_j} \times F_{\bar{y}_j}
$$
gives a splitting
$$ \widehat{J}(C)_1 = \bigoplus_j [\widehat{J}(C)_1]_{y_j} \oplus
[\widehat{J}(C)_1]_{\bar{y}_j}, $$
and we fix our labeling of the prime $y_1$ so that
$$ [\widehat{J}(C)_1]_{y_1} = \widehat{J}(C)[1]. $$
We then have
\begin{align*}
[\widehat{J}(C)_1]_{y_j} & = \widehat{J}(C)[t_j], \\
[\widehat{J}(C)_1]_{\bar{y}_j} & = \widehat{J}(C)[(p-1)-t_j]. 
\end{align*}

By the Honda-Tate classification (see, for instance, \cite[Ch.2]{taf}), we
deduce from the slopes of $\widehat{J}(C)$ that $J(C)_1$ is simple, and
hence 
$$ D_C = \End^0(J(C)_1) $$
is a central division algebra over $F$ whose only non-trivial invariants
are given by 
\begin{align*}
\Inv_{y_j} D_C & = t_j/(p-1), \\
\Inv_{\bar{y}_j} D_C& = -t_j/(p-1).
\end{align*}
Thus we see that $D_C$ is isomorphic to the division algebra $D$ 
constructed in
Section~\ref{sec:global}.  

Moreover, in the notation of Section~\ref{sec:involutions}, the subalgebra
$D'$ is equal to $D$.  The involution $\dag'$ is \emph{characterized} by the
property that $G_1$ is contained in $U_{(D,\dag')}(\QQ)$.  The abelian
variety 
$J(C)$ possesses a canonical polarization $\lambda_C$ coming from the Jacobian
structure.  Since the action of $G_1$ on $J(C)$ is induced from an action
on $C$, the action of $G_1$ preserves the polarization.  In particular, the
polarization restricts to a polarization 
$$ \lambda_C : J(C)_1 \rightarrow J(C)_1^\vee. $$
Letting $\dag_C$ be
the associated Rosati involution on $D_C$, we conclude that $G_1$ is
contained in $U_{(D_C, \dag_C)}(\QQ)$.  We therefore deduce that there is
an isomorphism
\begin{equation}\label{eq:dagC}
(D,\dag) \cong (D_C, \dag_C). 
\end{equation}

Specializing to the case of $p \in \{ 5,7 \}$, the field $F$ is a quadratic
imaginary extension of $\QQ$.  We have $k = 1$, $(p) = y_1 \bar{y}_1$, and 
$$ \widehat{J}(C)_1 = \widehat{J}(C)[1] \oplus \widehat{J}(C)[p-2]. $$
Fixing the Shimura data as in Section~\ref{sec:shaut}, it follows from
Theorem~\ref{thm:shaut} and (\ref{eq:dagC}) that there exists a
compact open subgroup $K \subset GU(\AF^{p,\infty})$ and a level structure
$\eta_0$ so that $C$ gives rise to a point
$$ (J(C)_1, i_C, \lambda_C, [\eta_0]_K) \in \Sh^{[n]}(\FF_p) $$
in height $n$ locus of the Shimura stack $\Sh(K)$.

\subsection{More Shimura stacks}\label{sec:moreShimura}

In Section~\ref{sec:negative}, we showed that if $p$ is odd, $n =
(p-1)mp^{\alpha-1}$ and $G_\alpha$
is an automorphism group of a height $n$ point of a Shimura stack $\Sh(K)$,
then we must have $m = 1$ and $p \le 7$.  However this analysis was
restricted to the class of Shimura stacks considered in \cite{taf} that
give rise to cohomology theories $\TAF$.  If one removes the restriction
that the Shimura stack has an associated cohomology theory, one can
extend the analysis of Section~\ref{sec:shaut} to show that there do exist
Shimura stacks with a mod $p$ point whose associated quasi-endomorphism
ring with involution does correspond to the pair $(D,\dag)$ of
Proposition~\ref{prop:dag}.  

Specifically, using the notation of Section~\ref{sec:global}, for arbitrary
$p$, $m$, and $\alpha$, let $B =
M_n(F)$.  Then there exists an involution $\iota$ on $B$ whose signatures
at the real places of $F^+$ are given by $(t_i, n-t_i)$, and a compact open
subgroup $K \subset GU_{(B,\iota)}(\AF^{p,\infty})$ such that the
associated Shimura stack has $G_\alpha$ as the automorphism group of a mod
$p$ point of the associated Shimura stack $\Sh_{(B,\iota)}(K)$ over
$\mc{O}_{F^+, x_1}$.

The reduction $\Sh_{(B,\iota)}(K) \otimes \bar{k}_{x_1}$ (where $k_{x_1}$
is the residue field of $F^+$ at $x_1$) possesses a stratification governed
by the Newton polygons of the associated $p$-divisible groups.  It is
possible that if one chooses a suitable stratum, one could associate to it
a cohomology theory via Lurie's theorem.  
If this is the case, setting $\alpha = 1$, one could use the
deformation theory of points in this stratum to give deformations of the
Jacobians of the Artin-Schreier curves studied by Ravenel in
\cite{RavenelAS}.

\subsection{Potential Applications to $EO_n$-resolutions}

In \cite{GHMR}, a resolution of the $K(2)$-local sphere at the prime $3$ is
constructed
$$ \ast \rightarrow S_{K(2)} \rightarrow X_0 \rightarrow \cdots \rightarrow
X_4 \rightarrow \ast $$
where the spectra $X_i$ are wedges of spectra of the form $E_2^{hG}$ for
various finite subgroups of the extended Morava stabilizer group.  In
\cite[Thm.~26]{Henn}, Henn extends this to show that for $p$ odd and $n = p-1$, there
is a similar resolution of $S_{K(p-1)}$ involving the spectra $EO_{p-1}$.

In \cite{Behrens}, the first author gave a moduli theoretic description of
\emph{half} of
the resolution of \cite{GHMR}.  In \cite{Behrensbldg}, the resolution of
\cite{Behrens} is shown to $K(2)$-locally arise strictly from the structure
of the quaternion algebra of quasi-endomorphisms of a supersingular
elliptic curve.  These constructions are further generalized in \cite{taf}.  
Even in the cases where there is no Shimura stack giving rise to the groups
$G_\alpha$, the results of Section~\ref{sec:global} and
Section~\ref{sec:involutions} still produce a very explicit presentation of a global
division algebra with involution $(D,\dag)$ 
containing the groups $G_\alpha$.  The buildings associated to the groups
$$ GU_{(D,\dag)}(F^+_\lambda) \cong GL_n(F^+_\lambda) $$ 
for primes $\lambda$ of $F^+$ not dividing
$p$ which split in $F$ will therefore produce length $n$ resolutions involving
$E_n^{hG_\alpha}$.  Analogs of the density results of
\cite{BehrensLawsondense},
\cite{Naumann} should help analyze to what degree these resolutions
approximate $S_{K(p-1)}$.

\subsection{Connective covers of $EO_n$}

One of the benefits to the equivalence
$$ \tmf_{K(2)} \simeq EO_2^{h\Gal(\bar{\FF}_p/\FF_p)} $$
is that the connective spectrum $\tmf_p$ serves as a well-behaved
connective cover for $EO_2$ for $p \in \{2,3\}$.  The connective spectrum
$\tmf$ in turn comes from the compactification of the moduli stack of
elliptic curves.  

One of the original motivations of the authors to investigate the results
of this paper was to try to give similar connective covers for $EO_n$.
The Shimura stacks associated to the group $GU$ of Section~\ref{sec:shaut}
possess similar compactifications.  In fact, these compactifications involve
adding only finitely many points in the locus where the associated formal
group has height $1$.  Therefore, like the case of $\TMF$, the construction
of connective forms of the associated $\TAF$ spectra is basically a
$K(1)$-local problem.  The main 
difficulty lies in the fact that in the equivalence
$$ \TAF_{GU}(K)_{K(n)} \simeq 
\left(
\prod_{[g] \in GU_{\mbf{A}_0}(\ZZ_{(p)})
\backslash GU(\AF^{p,\infty})/K}
E_n^{h\Gamma(gKg^{-1})}
\right)^{h\Gal(\bar{\FF}_p/\FF_p)}.
$$
the number of terms in the product could be greater than $1$.  
The computation of the cardinality of 
$$ GU_{\mbf{A}_0}(\ZZ_{(p)})
\backslash GU(\AF^{p,\infty})/K $$
is a class number question for the group $GU$.  Unfortunately, preliminary
calculations of this cardinality by the first author, using the mass
formulas of \cite{GHY}, seem to indicate
that, even in the case of $p = 5$ and $n = 4$, this class number is quite
large.

\subsection{Non-orientability of $TAF$-spectra}

Atiyah, Bott, and Shapiro \cite{ABS} 
showed that $\mit{Spin}$-bundles are $KO$-orientable,
thus giving an orientation
$$ \widehat{A}: M\mit{Spin} \rightarrow KO. $$
refining the $\widehat{A}$-genus.  The second author, with Ando and Rezk
\cite{AHR}, showed that the Witten genus refines to an orientation
$$ M\mit{String} \rightarrow \tmf. $$
One can ask the following question: is there a natural
topological group $G$ over $O$ for which there exists an orientation
$$ MG \rightarrow \TAF_{GU}(K)? $$
Na\"ively, since 
\begin{align*}
B\mit{Spin} & = BO\bra{4}, \\
B\mit{String} & = BO\bra{8}, 
\end{align*}
one might expect that one can take $G$ so that $BG = BO\bra{N}$ for $N$
sufficiently large.  However, at least if $p \in \{5,7\}$, $n = p-1$, and
$GU$ and $K$ are as in Theorem~\ref{thm:shaut},
the existence of a hypothetical $MO\bra{N}$-orientation of $\TAF_{GU}(K)$
would result in a composite of maps of ring spectra
$$ MO\bra{N} \rightarrow \TAF_{GU}(K) \rightarrow EO_{p-1}. $$
Hovey showed that no such composite can exist, for any $N$ and $p \ge 5$
\cite[Prop~2.3.2]{Hovey}.  Thus, the results of this paper imply that, at
least in some cases, the spectrum $\TAF_{GU}(K)$ is \emph{not} orientable by a
any connective cover of $O$.

\begin{rmk}
Hovey's non-orientability result relies only on the spectrum detecting the
$p$-primary $\alpha_1$ for $p > 3$ in its Hurewitz image.  Thus, the
comments in the following section actually imply the non-orientability of
a much larger class of $\TAF$-spectra by connective covers of $O$.
\end{rmk}

\subsection{$p$-torsion in automorphism groups at other primes}

Fix $p$ odd and set $n = (p-1)p^{\alpha-1}m$.
The results of Section~\ref{sec:negative} imply that in most cases, the
maximal finite group $G_\alpha$ cannot be realized as an automorphism group
of a mod $p$ point of height $n$ in one of the Shimura stacks under
consideration.  This does not preclude the possibility that these
automorphism groups could contain large $p$-torsion.  The main interest in
the groups $G_\alpha$ in homotopy theory 
is not that they are maximal, but rather that they contain an element of
order $p^\alpha$.

Indeed, fix $p$ to be any prime, and set $n = (p-1)p^{r-1}$.  Let $F$ be any
quadratic imaginary extension of $\QQ$ in which $p$ splits as $u\bar{u}$, and 
let $M =
F(\zeta)$, where $\zeta$ is a primitive $p^r$th root of unity.  We
necessarily have $[M:F] = n$.  Let $A_0/\bar{\FF}_p$ 
be an abelian variety of dimension
$n$ with complex
multiplication
$$ i_0 : F \rightarrow \End^0(A) := D $$
with associated $p$-divisible summand $A(u)$ of slope $1/n$ (such an
abelian variety exists by Honda-Tate theory).  Embed the
extension $M$ of $F$ in
$D$.  This gives $A_0$ complex multiplication by $M$.
Pick a
polarization $\lambda_0$ compatible with this $M$-linear structure: 
one exists by
\cite[Lem.~9.2]{Kottwitz}.  By the definition of compatibility, the
$\lambda_0$-Rosati involution restricts to the CM involution on $M$.  In
particular there is an inclusion
$$ \mu_{p^r} \hookrightarrow \Aut(A_0, i_0, \lambda_0). $$
Let $(\td{A}_0, \td{i}_0,
\td{\lambda}_0)/\CC$ 
be the base-change to the
complex numbers of a lift of $(A_0, i_0, \lambda_0)$ to characteristic $0$.
Following the argument of \cite[14.3.2]{taf}, the tuple 
$(\td{A}_0, \td{i}_0, \td{\lambda}_0)$ arises as a point of a complex Shimura
stack $\Sh_\CC$ associated to a group $GU_{(C,\iota)}$ where $C = M_n(F)$ and
$\iota$ is an involution associated to a hermitian form of signature
$(1,n-1)$.  Let $\Sh$ be the $p$-integral model of $\Sh_\CC$: by
construction the point $(A_0, i_0, \lambda_0)$ is a mod $p$ point, with an
automorphism group containing an element of order $p^r$.  We therefore have
established the following proposition.

\begin{prop}\label{prop:highorder}
Suppose that $p$ is any prime, and that $n = (p-1)p^{r-1}$.  Then for each
quadratic imaginary extension $F$ of $\QQ$ in which $p$ splits, there
exists an $n$-dimensional hermitian form of signature $(1,n-1)$, 
with associated unitary similitude group $GU$, and a compact open subgroup $K
\subset GU(\AF^{p,\infty})$, so that the associated Shimura stack contains
a point
$$ (A_0, i_0, \lambda_0, [\eta_0]_K) \in \Sh^{[n]}(K)(\bar{\FF}_p) $$
whose automorphism group contains an element of order $p^r$.
\end{prop}

\subsection{The prime $2$}

Suppose now that $p = 2$, and that $n = 2^{r-1}$ with $r > 2$.  By
\cite[Cor.~1.5]{Hewett}, every maximal finite subgroup of $\MS_n$ is
cyclic.  Let $\mu_{2^r} \subset \MS_n$ be the cyclic subgroup of order
$2^r$ arising from embedding the field $\QQ_2(\zeta_{2^r})$ in the division
algebra $D_{1/n}$ over $\QQ_2$ of invariant $1/n$.  Let $G$ be a maximal finite
subgroup in $\MS_n$ containing $\mu_{2^r}$.  Since $G$ is cyclic, there
must be a corresponding cyclotomic extension of $\QQ_2$ containing 
$\QQ_2(\mu_{2^r})$
which embeds in $D_{1/n}$.  Since $\QQ_2(\zeta_{2^r})$ is a maximal
subfield, we conclude that $G = \mu_{2^r}$, and that $\mu_{2^r}$ is a
maximal finite subgroup.  Using Proposition~\ref{prop:highorder}, we 
therefore have the following $2$-primary version of Theorem~\ref{thm:shaut}.

\begin{prop}
Suppose that $p = 2$ and that $n = 2^{r-1}$ with $r > 2$.  Then for each
quadratic imaginary extension $F$ of $\QQ$ in which $2$ splits, there
exists an $n$-dimensional hermitian form of signature $(1,n-1)$, 
with associated unitary similitude group $GU$, and a compact open subgroup $K
\subset GU(\AF^{2,\infty})$, so that the associated Shimura stack contains
a point
$$ (A_0, i_0, \lambda_0, [\eta_0]_K) \in \Sh^{[n]}(K)(\bar{\FF}_2) $$
whose automorphism group is the maximal subgroup $\mu_{2^{r}} \subset
\MS_n$.
\end{prop}

\providecommand{\bysame}{\leavevmode\hbox to3em{\hrulefill}\thinspace}
\providecommand{\MR}{\relax\ifhmode\unskip\space\fi MR }
\providecommand{\MRhref}[2]{%
  \href{http://www.ams.org/mathscinet-getitem?mr=#1}{#2}
}
\providecommand{\href}[2]{#2}

\end{document}